\documentclass[11pt]{amsart}

\usepackage{amsmath}
\usepackage{amssymb}
\usepackage{graphicx}

\newtheorem{theorem}{Theorem}[section]

\newtheorem{lemma}[theorem]{Lemma}
\newtheorem{proposition}[theorem]{Proposition}
\theoremstyle{definition}
\newtheorem{definition}[theorem]{Definition}
\newtheorem{example}[theorem]{Example}
\newtheorem{remark}[theorem]{Remark}

\numberwithin{equation}{section}

\DeclareMathOperator*{\argmin}{argmin}
\usepackage{mathrsfs}

\usepackage{algorithm}
\usepackage{algorithmic}

\newcommand{\ttop}{\mathsf{T}}
\newcommand{\x}{x}
\newcommand{\y}{y}
\newcommand{\z}{z}
\newcommand{\h}{h}
\newcommand{\vv}{v}
\newcommand{\dd}{d}
\newcommand{\uu}{u}
\newcommand{\g}{g}
\newcommand{\p}{p}
\newcommand{\q}{q}

\usepackage{algorithm, algorithmic}

\makeatletter
\renewcommand{\p@algorithm}{\arabic{algorithm}\expandafter\@gobble}
\makeatother 
\newcommand{\PARAMETERS}{\item[\textbf{Parameters:}]}

\newcounter{step}[algorithm]
\newcommand\STEP[2][\(\triangleright\)]{%
	\refstepcounter{step}
	\vskip 0.25\baselineskip
	\item[]\hskip -\algorithmicindent #1 \textbf{Step \arabic{step}}%
	\ifthenelse{\equal{\unexpanded{#2}}{}}{}{ (\texttt{#2})}%
	\textbf{.}%
}
\newenvironment{algo}{\algo}{}
\def\algo#1\end{%
	\noindent\fbox{%
	\begin{minipage}[b]{\dimexpr\columnwidth-\algorithmicindent\relax}
	\begin{algorithmic}
	#1
	\end{algorithmic}
	\end{minipage}
	}%
\end}

\usepackage{auto-pst-pdf}
\usepackage{pst-plot}
\usepackage{pdftricks}
\begin{psinputs}
\usepackage{pstricks,pst-plot}
\end{psinputs}

\title[Nonsmooth trust-regions]{Cutting plane oracles for non-smooth trust-regions}

\author[D. Noll]{Dominikus Noll}

\address[D. Noll]{Universit\'e de Toulouse, Institut de Math\'ematiques, France}
\email{{\tt dominikus.noll@math.univ-toulouse.fr}}



\begin{document}

\begin{abstract}
We prove global convergence of a bundle trust region algorithm for non-smooth non-convex optimization,
where cutting planes are generated by oracles
respecting four basic rules. The benefit  is that convergence theory applies to a large variety of
methods encountered in practice. This includes in particular the method of downshifted tangents,
for which previously no convergence result in the trust region framework was known. We also show that certain splitting techniques can be seen as special cases
of bundle trust region techniques. 
\end{abstract}

\maketitle


\section{Introduction}
We consider optimization  problems of the form
\begin{equation}
\label{program}
\min_{\x \in C} f(\x)
\end{equation}
where $f:\mathbb R^n \to \mathbb R$ is locally Lipschitz, but neither differentiable, nor convex, and where $C$ is a closed convex set in $\mathbb R^n$, which is typically
of a simple structure, such as a polyhedron, or a semi-definite set. 
We discuss a non-smooth  trust region algorithm, which 
generates local models of $f$ by accumulating information from cutting planes at unsuccessful trial steps. We prove global convergence of our algorithm
in the sense that all accumulation points $\x^*$ of the sequence of
serious iterates $\x^j$  are critical points of (\ref{program}), i.e., satisfy $0\in \partial f(\x^*)+N_C(\x^*)$.  

In \cite{ANR:2016,ANR:18} we had discussed a trust region method,
where model-based cutting planes 
were used.  While this type of oracle is suited for a number of applications, we presently cover a larger variety
of oracles encountered in practice. This includes in particular
the well-know method of downshifted tangents, for which convergence in the  trust regions framework had so far not been established.
Our global convergence result applies to all 
cutting plane oracles   satisfying four basic rules.

We answer the question posed in \cite{rus},  whether aggregation  developed in \cite{kiwiel_aggregation} for the convex bundle method 
to limit the number of cuts,
could also be used in trust regions for the same purpose. In \cite{NPR:08} we had justified aggregation for the non-convex bundle method,
and  in \cite{ANR:18} for model based oracles in trust regions under the additional hypothesis that a positive definite second-order
term was used in the working model. Here we complete our analysis by showing that this hypothesis cannot be removed.
See
 \cite[Remark 10]{ANR:2016} and also
\cite{zowe,zowe2,hare,saga} for this topic. 

Convergence of bundle or trust region bundle methods requires that the objective $f$
admits a {\em strict model}. Models of locally Lipschitz functions were introduced in \cite{NPR:08} and may be understood as non-smooth substitutes for the first-order Taylor expansion.
Every model gives rise to a cutting plane oracle, hence to a trust region or bundle method, and conversely, every cutting plane oracle generates a model as its upper envelope, so 
both notions
are closely related.
Strictness of a model in the non-smooth setting is the
analogue to strict differentiability in the smooth setting. Strictness of an oracle decides whether or not the 
corresponding bundle or trust region method converges.

While the present work concentrates on theoretical convergence aspects, we
 have applied non-convex non-smooth bundle and trust region methods successfully
to applications in mechanical contact problems, control and system theory, and operations research. For instance
\cite{gabarrou} applies a bundle technique with downshift to design a flight controller for a civil aircraft.  In \cite{eigen}
a model-based bundle method is used for partial eigenstructure assignment, with applications
to a launcher in atmospheric flight and to decoupling motions of a civil aircraft.
In \cite{IJRNC,ANR:18,RNRB} the model-based trust-region method is applied to
$H_\infty$-control of infinite-dimensional systems, including boundary control of PDEs.  In \cite{gwinner}
we use our model-based technique in a delamination study in destructive testing in the material sciences.
A non-smooth approach to reliable flight control is  \cite{encyc}, where we also give a survey of recent relevant non-smooth methods in control. 
In \cite{memory} a bundle method with downshift is 
used to minimize the memory stored in a system, and an evacuation scenario of a fairground in case of an emergency 
is solved. In \cite{ANP:2008,ANP:2009} a non-convex spectral bundle method is used  to solve bilinear matrix inequalities.
This corresponds to cases where an infinity of cutting planes arise simultaneously. In \cite{simoes} a frequency shaping control
technique is developed using  model-based bundling.  Details on how non-linear constraints are handled are given in \cite{gabarrou,simoes,encyc}.
The first non-convex bundle method with convergence certificate under inexact subgradients and function values is \cite{inexact}, where downshifted tangents are used.
The model-based trust-region bundle method has been used for
parametric and mixed robust control, one of the most
challenging problems in feedback control design \cite{ANR:2016,AN:16,ADN:2015,AAN:2017}.   We have also used the trust region bundle method within a 
branch-and-bound approach to global optimization problems like computing the distance to instability of a controlled system,
see \cite{RNA:17,siam}.  All these techniques can be seen as special cases of the present
abstract framework. 


The structure of the paper is as follows. In section \ref{sect_model} we discuss 
when a function $f$ admits a strict model. Section \ref{sect_oracle} introduces rules for cutting plane oracles and relates them to models.
We recall how cutting planes are used to build working models.
Section \ref{optimize} presents the bundle trust region algorithm and  its elements. A major difference with classical trust regions occurs 
in our management of the trust region radius. The central section \ref{convergence}  gives the global convergence proof for the algorithm. 
In section \ref{applications} we show that splitting techniques, the proximal point method, but also classical
gradient oriented and quasi Newton methods  are special cases of our oracle based bundle trust region concept.   
Extensions to constraint programs $\min\{f(\x): c(\x)\leq 0, \x\in C\}$ can be obtained by similar techniques via suitable 
progress functions in the sense of  \cite{polak2,mifflin2,mifflin3},  or \cite{gabarrou,simoes,encyc,dao}, or by multi-objective methods \cite{simoes,multidisk,multiband}.

\section{The model concept}
\label{sect_model}
The basic tool in bundle and bundle trust-region methods is
the cutting plane. Its rationale \cite{rus,lemarechal1,lemarechal2} is that if an unsuccessful trial step $\y$ (called a	 null step)  in the neighborhood of the current serious iterate $\x$ is made,
then a cutting plane to $f$ at $\y$ should be included in the working model in order to orient the search for a new $\x^+$ away from $\y$ in the next trial. 
In the convex case cuts are simply tangents to $f$ at $\y$, but for non-convex $f$, different ways to obtain cuts are needed.
In order to understand
how cutting planes are generated, we shall need 
the notion of a {\em model} of $f$ near an iterate $\x$, which can be considered as a non-smooth analogue of the first-order Taylor expansion of $f$ at $\x$. 
In this chapter, the theoretical background of the model concept is recalled.

\subsection{Model axioms}
Let
$f:\mathbb R^n \to \mathbb R$ be a locally Lipschitz function. Following \cite{NPR:08,cutting},  a function
$\phi:\mathbb R^n \times \mathbb R^n \to \mathbb R$ is called a {\em model} of $f$ if it satisfies the following
properties:
\begin{itemize}
\item[$(M_1)$] $\phi(\cdot,\x)$ is convex, $\phi(\x,\x) = f(\x)$, and $\partial_1\phi(\x,\x) \subset \partial f(\x)$.
\item[$(M_2)$] If $\y_k\to \x$, then there exist $\epsilon_k\to 0^+$ such that $f(\y_k) \leq \phi(\y_k,\x)+ \epsilon_k\|\y_k-\x\|$.
\item[$(M_3)$] If $\x_k\to \x, \y_k\to \y$, then $\displaystyle\limsup_{k\to\infty} \phi(\y_k,\x_k) \leq \phi(\y,\x)$.
\end{itemize}

\begin{remark}
We may interpret $\phi(\cdot,\x)$ as a substitute for the first-order Taylor expansion of  $f$ at $\x$.  
This is highlighted by the fact that every locally Lipschitz function $f$ has the so-called  {\em standard model}
\[
\phi^\sharp(\cdot,\x) = f(\x) + f^\circ(\x,\cdot - \x),
\]
where $f^\circ(\x,\dd)$ is the Clarke directional derivative at $\x$ in direction $\dd$. However,
a function $f$ may have several models, so the notion of a  model, unlike the Taylor expansion, claims no uniqueness.
\end{remark}

\begin{remark}
Composite functions $f = \psi \circ F$ with convex $\psi$ and  $F$ of class $C^1$
admit the so-called {\em natural model} $\phi(\cdot,\x) = \psi\left(  F(\x) + F'(\x)(\cdot - \x) \right)$, see \cite{ANP:2008,ANP:2009,NPR:08,cutting}.
Consequently, every lower-$C^2$ function admits a strict model, because by \cite{RW98} it can be represented in the form $f = \psi \circ F$ on every bounded set.
\end{remark}


\begin{remark}
When $f$  is convex,  its natural model is $\phi(\cdot,\x) = f$.
We say that a convex function is its own natural model. Note that convex $f$ have still their standard model $\phi^\sharp$, which
is different from $f$ unless $f$ is affine. 
\end{remark}



\begin{proposition}
\label{smallest}
{\rm (See \cite[Lemmas 1,2]{cutting}).}
The standard model $\phi^\sharp$ is the smallest model of $f$, i.e., $\phi^\sharp \leq \phi$ for any other model $\phi$ of $f$.
\hfill $\square$
\end{proposition}

\begin{remark}
\label{idea}
The following operation on models is sometimes useful. Suppose $f$ has two models, $\phi_1,\phi_2$,
then $\phi = \max\{\phi_1,\phi_2\}$ is again a model.  Equally useful is the following. Suppose
$\phi_1,\dots,\phi_m$ satisfy $(M_1)$ and $(M_3)$. Moreover, suppose $\phi_i$ satisfies $(M_2)$ at points $\x\in \Omega_i$
such that $\Omega_1 \cup \dots \cup \Omega_m = \mathbb R^n$. Then $\phi = \max\{\phi_1,\dots,\phi_m\}$ is a model of $f$. 
\end{remark}

\begin{remark}
There is another operation, which is more problematic. Suppose $f = f_1 + f_2$ where $f_i$ has model $\phi_i$. Then it would be convenient were
$\phi = \phi_1 + \phi_2$ a model of $f$. This fails in general due to axiom $(M_1)$. Indeed,
we have $\partial_1 \phi(\x,\x) = \partial_1 \phi_1(\x,\x) + \partial_1 \phi_2(\x,\x) \subset \partial f_1(\x) + \partial f_2(\x)$, but 
unfortunately, $\partial f_1(\x) + \partial f_2(\x) \not\subset \partial f(\x)$ in general. In this situation we can modify the model concept as follows:
We use 
$\partial^T \! f(\x) := \partial f_1(\x) + \partial f_2(\x)$ in axiom $(M_1)$ instead of the Clarke subdifferential $\partial(f_1+f_2)(\x)$.  This is tolerable since every Clarke 
critical point satisfies also $0\in \partial^T\! f(\x)$, i.e., is critical in this extended sense. Toland \cite{toland1,toland2} was the first to use
this type of subdifferential for dc-functions, i.e.,  $f = g-h$ with $g,h$ convex, which is why we use the notation $\partial^T$. 
\end{remark}


\subsection{Strict models}
\label{sect_strict}

For convergence
theory we will need a slightly stronger type of model, which is given  by the following:

\begin{definition}
{\rm
A  model $\phi$ of $f$ is called {\em strict} at $\x\in \mathbb R^n$ if it satisfies the following strict version of axiom $(M_2)$:
\begin{itemize}
\item[$(\widehat{M}_2)$] If $\x_k,\y_k\to \x$, then there exist $\epsilon_k\to 0^+$ such that $f(\y_k) \leq \phi(\y_k,\x_k)+ \epsilon_k\|\y_k-\x_k\|$.
\end{itemize}
The model $\phi$ is called strict if it is strict at every $\x$.}
\hfill $\square$
\end{definition}

\begin{remark}
The difference between $(M_2)$ and the strict version $(\widehat{M}_2)$ is analogous to the difference between 
differentiability and strict differentiability, hence the nomenclature.
\end{remark}

One may ask which locally Lipschitz functions $f$ admit strict models. Since every $f$ has its standard model $\phi^\sharp$,  it is natural
to ask first
whether or when $\phi^\sharp$ is strict.

Recall that a locally Lipschitz function $f$ is upper $C^1$ at $\bar{x}$
if for every $\epsilon > 0$ there exists $\delta > 0$ such that for all $\x,\y\in B(\bar{\x},\delta)$ and {\em every} $\g \in \partial f(\x)$
one has $f(\y) \leq  f(\x) + \g^\ttop (\y-\x) + \epsilon \|\y-\x\|$, where  the latter could also be written as
\[
f(\y) \leq f(\x) -f^\circ(\x,\x-\y) + \epsilon \|\y-\x\|.
\]
We weaken this by saying that $f$ is weakly upper $C^1$ at $\bar{\x}$ if for every $\epsilon > 0$ there
exists $\delta > 0$ such that for all $\x,\y\in B(\bar{\x},\delta)$ and {\em some} $\g\in \partial f(\x)$ one has
$
f(\y) \leq f(\x) + \g^\ttop (\y-\x) + \epsilon \|\y-\x\|,
$
or what is the same
\[
f(\y) \leq f(\x) + f^\circ(\x,\y-\x) + \epsilon \|\x-\y\|.
\]
That is precisely strictness of $\phi^\sharp$, so we have (see \cite{loja}):

\begin{proposition}
\label{upper_strict}
If $f$ is weakly upper $C^1$ at $\x$, then
the standard model $\phi^\sharp$ of $f$ is strict at $\x$.  In that case, every model $\phi$ of $f$ is strict at $\x$.
\hfill $\square$
\end{proposition}


\begin{remark}
A consequence is that if
$f$ is differentiable at $\x$,  then the standard model is strict at $\x$ if and only if $f$ is strictly differentiable at $\x$. 
\end{remark}

\begin{example}
\label{zig_zag}
Expanding on this, suppose $f$ is differentiable but not strictly so, can it still have a strict model $\phi$ other than $\phi^\sharp$? 
For instance, can $f(\x)=\x^2 \sin \x^{-1}$ have a model which is everywhere strict, including the origin?
The answer is surprisingly 'yes'. 

We consider for technical reasons a discrete version of $\x^2(\sin\x^{-1}+1)$.
Define $f:[-1,1] \to \mathbb R^+$ as follows: start at $t_1=1$ with value $f(t_1)=0$, then choose slope
$-1$ on $[t_2,t_1]$, where $t_2 = \frac{\sqrt{5}-1}{2}$ with value $f(t_2)=t_2^2$.  Next let $f$ have slope $+1$ on $[t_3,t_2]$ with value $f(t_3)=0$, etc.
This leads to the recursions
\[
t_{2k+1} = t_{2k}-t_{2k}^2, \quad t_{2k} = \frac{\sqrt{1+4t_{2k-1}}-1}{2}, \quad f(t_{2k})=t_{2k}^2, \quad  f(t_{2k+1})=0,
\]
with $f$ piecewise linear with slope $\pm 1$ in between the $t_i$. In fact, $f$ zig-zags between the axis $y=0$ and the parabola $y=x^2$. We extend the function symmetrically to $[-1,0]$. 
Note that $f$ is differentiable at the origin with $f'(0)=0$, but is not strictly differentiable, because $\partial f(0)=[-1,1]$.

We show that the standard model $\phi^\sharp$ of $f$ is not strict. Suppose it were, then
$f(t_{2k}) {\leq} \phi^\sharp(t_{2k},t_{2k+1}-\delta_k) + {\rm o}(t_{2k}-t_{2k+1}+\delta_k)$, where
$\delta_k > 0$ with $t_{2k+1} - \delta_k > t_{2k+2}$. Since on $[t_{2k+2},t_{2k+1}]$ the slope is $-1$, this amounts to 
$t_{2k}^2 {\leq} f(t_{2k+1}-\delta_k) + (-1) (t_{2k}-t_{2k+1}+\delta_k) + {\rm o}(t_{2k}-t_{2k+1}+\delta_k)= \delta_k + t_{2k+1}-t_{2k} - \delta_k + {\rm o}(t_{2k}-t_{2k+1}+\delta_k)=-t_{2k}^2+ {\rm o}(t_{2k}-t_{2k+1}+\delta_k)$. So we would have $2t_{2k}^2 \leq {\rm o}(t_{2k}-t_{2k+1}+\delta_k)$, and since $\delta_k$ can be chosen arbitrarily small, we would need
$t_{2k}^2  \leq {\rm o}(t_{2k+1}-t_{2k})={\rm o}(t_{2k}^2)$, which is wrong. Hence $\phi^\sharp$ is not strict at $0$.  
\end{example}

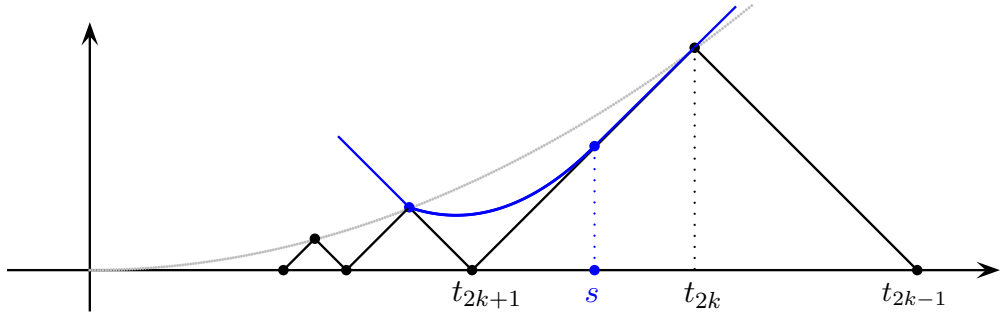
\hspace*{-1cm}
\begin{pspicture}(-2,-1)(6,4)

\psscalebox{1.1}{
\psline[arrows=->,arrowsize=5pt 1](-1,0)(11,0)
\psline[arrows=->,arrowsize=5pt 1](0,-.5)(0,3)

\psplot[linecolor=lightgray,plotstyle=dots,
      plotpoints=250,dotstyle=*,dotsize=1pt]
      {0}{8}{x 2 exp 20 div} 

\psdot(10,0)

\psline(10,0)(7.31,2.69)      
\psdot(7.31,2.69)

\psline(7.31,2.69)(4.62,0)

\psdot(4.62,0)
\psline(4.62,0)(3.86,0.76)

\psdot[linecolor=blue](3.86,0.76)
 
\psline(3.86,0.76)(3.1,0) 

\psdot(3.1,0)

\psline(3.1,0)(2.72,0.38)

\psdot(2.72,0.38)

\psline(2.72,0.38)(2.34,0)

\psdot(2.34,0)

\psdot[linecolor=blue](6.1,1.5)

\psplot[linecolor=blue,plotstyle=dots,
      plotpoints=250,dotstyle=*,dotsize=1pt]
      {3.86}{6.1}{x 2 exp 0.2989 mul 6.5221 add x 2.6466 mul sub} 
      
      \psline[linecolor=blue](6.1,1.5)(7.81,3.19)
      \psline[linecolor=blue](3.86,0.76)(3,1.62)
      
      \psdot[linecolor=blue](6.1,0)
      \psline[linestyle=dotted,linecolor=blue](6.1,0)(6.1,1.5)
      
      \uput[0](5.8,-.3){$\color{blue}s$}
      \uput[0](4.2,-.3){$t_{2k+1}$}

 \psline[linestyle=dotted,linecolor=black](7.31,0)(7.31,2.69)
 \uput[0](7.0,-.3){$t_{2k}$}
 
 \uput[0](9.4,-.3){\small $ t_{2k-1}$}
 
 }
 
\end{pspicture}

We now show that $f$ admits a strict model $\phi$. For peaks $(t_{2k},t_{2k}^2)$ let
$\phi(\cdot,t_{2k})= t_{2k}^2 + |\cdot - t_{2k}|$.  Then $\partial_1 \phi(t_{2k},t_{2k}) = [-1,1] = \partial f(t_{2k})$. The same for every bottom points
$t_{2k+1}$, including the origin. Hence again $\partial_1\phi(t_{2k+1},t_{2k+1})=[-1,1]=\partial f(t_{2k+1})$. For intermediate points we define $\phi$
as follows. Let $t_{2k+1} < s < t_{2k}$ (as in the figure). Then
$\phi(\cdot,s)$ consists of three arcs (shown in blue). Fit a parabola through the points $(t_{2k+2},t_{2k+2}^2)$ and $(s,f(s))$ such that its slope at $s$ equals $f'(s)=1$. To the left
of $t_{2k+2}$ extend by a line with slope $-1$, to the right of $s$ by a line of slope 1. This is a continuous convex function, which is differentiable everywhere, except at 
$t_{2k+2}$. Moreover, $\phi(\cdot,s) \geq f$. For $s \in (t_{2k+2},t_{2k+1})$ we proceed symmetrically, where
now the non-differentiability occurs at $t_{2k}$. Note that $\phi(\cdot,s)$ depends continuously on $s$ for $s\in (t_{2k+2},t_{2k})$,
where at $s=t_{2k+1}$ the three arc function degrades to the vee-shaped function $\phi(\cdot,t_{2k+1})=|\cdot-t_{2k+1}|$. Only at the peaks $s=t_{2k}$ is the construction discontinuous, but it remains
upper semi-continuous, so that $(M_3)$ is satisfied at the $t_{2k}$. Finally, we also observe that we have upper semi-continuity at $s=0$. This assures $(M_3)$ also at $0$. 
Strictness $(\widehat{M}_2)$ follows because $f \leq \phi(\cdot,s)$.

The function $f$ is not weakly upper semi-smooth in the sense of \cite{mifflin2,mifflin3} at $0$, nor is it upper $C^1$, so it is not amenable to linesearch methods. 
Yet the fact that it admits a strict model shows that it is amenable to non-smooth optimization
techniques in our framework. 
\hfill $\square$

\begin{remark}
Borwein and Moors \cite{borwein_densely_strict,borwein_moors} construct locally Lipschitz functions where subdifferential integrability fails. Consider in particular a function
$f:\mathbb R \to \mathbb R$ where $\partial f(\x) = [-1,1]$ for every $\x$. Then the standard model of $f$ is strict, but $f$ is not upper $C^1$. See also
\cite{borwein_zhu}.
\end{remark}

\begin{remark}
The natural model $\phi(\cdot,\x) = \psi(F(\x) + F'(\x) (\cdot - \x))$ of a composite function $f = \psi \circ F$ with $\psi$ convex and $F$ of class $C^1$ is strict, because $F(\y) = F(\x) + F'(\x)(\y-\x) + {\rm o}(\|\y-\x\|)$ as $\y-\x \to 0$, hence
$f(\y) - \phi(\y,\x) = \psi(F(\y)) - \psi(F(\x)+F'(\x)(\y-\x)) = {\rm o}(\|\y-\x\|)$, as $\psi$ is locally Lipschitz. In particular, every convex $f$ is its own strict model $\phi(\cdot,\x)=f$.
Another consequence is that every lower $C^2$-function has a strict model. \end{remark}

However, the following is more general:

\begin{proposition}
\label{lowerC1}
Every  lower $C^1$-function admits a strict model. 
\end{proposition}

A direct proof can be found in \cite{cutting}. Here we will obtain a slightly stronger result from our more general theory
of cutting plane oracles in section \ref{strictness_downshift}.

\begin{remark}
We call $f$ a dc-function  if it is the difference of two convex functions, i.e., 
$f = g - h$ for convex $g,h$. Let us call $f$ a DC-function if it admits a dc-decomposition $f = g - h$ where $\partial f(\x) = \partial g(\x) - \partial h(\x)$. Every DC-function has
the strict model $\phi(\cdot,\x) = g(\cdot) + \phi^\sharp(\cdot,\x)$, where $\phi^\sharp$ is the standard model of the concave function
$-h$. Note that if $f$ is dc but not DC, we can still use $\partial^T f(\x) = \partial g(\x) - \partial h(\x)$, then $\phi$ is automatically a strict $\partial^T$-model of $f$.
For more information on dc-functions see \cite{hiriart}.
\end{remark}

\begin{remark}
Consider the Euclidian distance  $f(\x) = \frac{1}{2}d_S^2(\x)$ to an arbitrary set $S\subset \mathbb R^n$. It follows with
\cite{asplund} that 
$h(\x)=\frac{1}{2}|\x|^2- \frac{1}{2}d_S^2(\x)$ is convex, hence $-f$ is lower $C^2$, and therefore $f$ is upper $C^2$. This means both $d_S^2$ and $-d_S^2$ have strict models.
For $f$ we may use the standard oracle, which leads to the steepest descent method. For $-f$ we use downshifted tangents.
Note that $f$ is also a DC-function, because $\frac{1}{2}d_S^2(\x) = \frac{1}{2}|\x|^2 - \left(  \frac{1}{2}|\x|^2- \frac{1}{2}d_S^2(\x) \right)$,
so  the previous remark gives  yet another strict model of $f$.
\end{remark}


\section{Cutting plane oracles}
\label{sect_oracle}
By a cutting plane oracle we understand, loosely, any procedure
$\mathscr O$, which associates with every serious iterate $\x$ and unsuccessful trial step (null step) $\z$ near $\x$ one -- or several -- affine
functions $m(\cdot,\x) = a + \g^\ttop (\cdot - \x)$,
which replace the tangent to $f$ at $\y$.
These cuts
are then accumulated to build a  convex working model $\phi_k(\cdot,\x)$ of $f$ in the neighborhood of $\x$.
In the convex cutting plane or bundle method, $\mathscr O(\z,\x)$ consists of any of the
tangents to $f$ at $\z$, i.e. $t_{\z,g}(\cdot) = f(\z) + \g^\ttop (\cdot - \z)$, where $\g \in \partial f(\z)$, so is in fact independent of the serious iterate $\x$. 
In the non-convex case we may no longer proceed in this way, because tangents to $f$ at $\z$ may pass above $f(\x)$, and cannot be used directly as cutting planes.

Let $\mathscr O:\mathbb R^n \times \mathbb R^n \rightrightarrows \mathbb R \times \mathbb R^n$ be a set-valued operator mapping into the nonempty bounded subsets of $\mathbb R \times \mathbb R^n$, 
where $(a,\g) \in \mathscr O(\z,\x)$ is understood as to represent the affine function $a + \g^\ttop (\cdot - \x)$. We call $\mathscr O$ 
a {\em cutting plane oracle} for $f$ if it
is bounded on bounded sets and satisfies
the following axioms:

\begin{itemize}
\item[$(\mathscr O_1)$] If $(a,\g)\in \mathscr O(\z,\x)$, then $a \leq f(\x)$. Moreover, 
$\mathscr O(\x,\x)$ contains at least one element $(f(\x),\g)$ with $\g \in  \partial f(\x)$.

\item[$(\mathscr O_2)$]
If $(a_j,\g_j)\in \mathscr O(\z_j,\x)$ with $a_j \to f(\x)$, $\|\z_j -\x\|\leq M$,  and $\g_j \to \g$, then $\g\in \partial f(\x)$.

\item[$(\mathscr O_3)$] If $\z_j \to \x$ there exist $(a_j,\g_j)\in \mathscr O(\z_j,\x)$ and
$\epsilon_j\to 0^+$ such that $f(\z_j) \leq a_j + \g_j^\ttop (\z_j-\x) + \epsilon_j \|\z_j-\x\|$.

\item[$(\mathscr O_4)$] If $\z_j \to \z$, $\y_j\to \y$ and $\x_j\to \x$,
and if $(a_j,\g_j)\in \mathscr O(\z_j,\x_j)$, then there exists $\z'\in B(\x,M)$ and $(a',\g')\in \mathscr O(\z',\x)$ such that
$\limsup_{j\to \infty} a_j + \g_j^\ttop (\y_j-\x_j) \leq a' + \g'^\ttop (\y-\x)$.
\end{itemize}
For the $M > 0$ occurring in $(\mathscr O_2),(\mathscr O_4)$ we define the envelope function of $\mathscr O$
by
\begin{equation}
\label{envelope}
\phi(\cdot,\x) = \sup\{a + \g^\ttop (\cdot-\x): (a,\g) \in \mathscr O(\z,\x), \|\z-\x\| \leq M\}.
\end{equation}

\begin{proposition}
Suppose $\mathscr O$ satisfies axioms {\rm $(\mathscr O_1)-(\mathscr O_4)$}.
Then the envelope function $\phi$ of $\mathscr O$ is a model of $f$.
\end{proposition}

\begin{proof}
1) Observe that $\mathscr O(\x,\x) \not=\emptyset$ by $(\mathscr O_1)$.  Pick $(a,\g)\in \mathscr O(\x,\x)$, then $\phi(\y,\x) \geq a +\g^\ttop (\y-\x)> -\infty$. This means
$\phi(\cdot,\x)$ is by construction convex and maps  into  $\mathbb R \cup \{\infty\}$.  But $\phi(\y,\x) = \infty$ is impossible, because  that would require a sequence $\z_j \in B(\x,M)$
and $(a_j,\g_j)\in \mathscr O(\z_j,\x)$ such that $a_j + \g_j^\ttop (\y-\x) \to \infty$. Then $(a_j,\g_j)$ would have to be unbounded, contradicting the fact that $\mathscr O$ is 
by definition bounded on the bounded set
$B(\x,M)\times \{\x\}$. This means $\phi(\cdot,\x)$ is a convex function which is everywhere defined.

2) We have to check $(M_1)$. Now by $(\mathscr O_1)$ we have $a \leq f(\x)$ for every $(a,\g) \in \mathscr O(\z,\x)$, hence
$\phi(\x,\x) \leq f(\x)$. But there exists $(f(\x),\g) \in \mathscr O(\x,\x)$, hence $\phi(\x,\x) \geq f(\x)$, giving equality. It remains to check
that every subgradient $\g$ of $\phi(\cdot,\x)$ at $\x$  belongs to $\partial f(\x)$.

Let $\g$ be a subgradient of $\phi(\cdot,\x)$ at $\x$, then $t(\cdot) := f(\x) + \g^\ttop (\cdot - \x)$ is a tangent to $\phi(\cdot,\x)$
at $\x$, and by convexity $t(\cdot) \leq \phi(\cdot,\x)$. Now fix $\h$ and consider the convex function $t \mapsto \phi(\x+t\h,\x)$ on the real line.
Its slope at $t=0$ is greater or equal than $\g^\ttop \h$. Now for every $t > 0$ find $\z_t \in B(\x,M)$ and a cutting plane $(a_t,\g_t)\in \mathscr O(\z_t,\x)$,
represented as
$m_t(\cdot,\x) = a_t + \g_t^\ttop (\cdot-\x)$, such that $\phi(\x+t\h,\x) \geq m_t(\x+t\h,\x)$ with equality $\phi(\x+t\h,\x) = m_t(\x+t\h,\x)$. Observe that this implies $a_t \to f(\x)$.

Now the slope of $\phi(\cdot,\x)$
at $\x+t\h$ is steeper than its slope at $\x$ (monotonicity), so $\g^\ttop \h \leq \g_t^\ttop \h$. Choosing a subsequence
$\g_t \to \hat{\g}$, $\z_t \to \z$, we get $\hat{\g} \in \partial f(\x)$ by axiom $(\mathscr O_2)$. That shows
$\g^\top \h \leq \hat{\g}^\ttop \h \leq \max\{\tilde{\g}^\ttop \h: \tilde{\g}\in \partial f(\x)\}= f^\circ(\x,\h)$. But since $\h$ was arbitrary,
we deduce $\g\in \partial f(\x)$ by Hahn-Banach. This completes the proof of $(M_1)$.

3) Since $(M_2)$ is clear from $(\mathscr O_3)$, it
remains to check $(M_3)$. Fix $\y_k\to \y$, $\x_k\to \x$, and let $\z_k\in B(\x_k,M)$ and $\epsilon_k \to 0^+$ such that
$m_{\z_k,\g_k}(\y_k,\x_k) = a_k + \g_k^\ttop (\y_k-\x_k) = \phi(\y_k,\x_k)$, where $(a_k,\g_k) \in \mathscr O(\z_k,\x_k)$.
Invoking $(\mathscr O_4)$, we get $\z'\in B(\x,M)$ and $(a',\g')\in \mathscr O(\z',\x)$ such that
$\limsup_{k\to\infty} a_k+\g_k^\ttop (\y_k-\x_k) \leq a' + \g^{'\ttop}(\y-\x) \leq \phi(\y,\x)$, and with the above this implies
$\limsup_{k\to\infty} \phi(\y_k,\x_k) \leq \phi(\y,\x)$.
\end{proof}

\begin{definition}
{\rm
A cutting plane oracle $\mathscr O$ is called {\em strict} at $\x$  if the following strict version of $(\mathscr O_3)$ is satisfied:
\begin{itemize}
\item[$(\widehat{\mathscr O}_3)$] Given $\x_j,\z_j \to \x$ there exist $(a_j,\g_j)\in \mathscr O(\z_j,\x_j)$ and $\epsilon_j\to 0^+$
such that $f(\z_j) \leq a_j + \g_j^\ttop(\z_j-\x_j) + \epsilon_j \|\z_j-\x_j\|$. 
\end{itemize}
\hfill $\square$
}
\end{definition}




\subsection{Model-based cutting plane oracles}
\label{model_oracle}
A natural way to generate cutting planes is when a model $\phi$ of $f$
is available for computation. Let $\x$ be the serious iterate, $\z$ an unsuccessful trial step (a null step) at which we want 
to generate a cutting plane. We simply take an affine support function $m_{\z,\g}(\cdot,\x)$ of   $\phi(\cdot,\x)$ at $\z$,
i.e. $m_{\z,\g} = \phi(\z,\x) + \g^\ttop (\cdot - \z)$ with $\g \in \partial_1 \phi(\z,\x)$. We use the notation
$\mathscr O_\phi$ for the oracle generated by a model $\phi$. If $\phi^\sharp$ is the standard model
of $f$, then we write $\mathscr O^\sharp := \mathscr O_{\phi^\sharp}$, calling it the {\em standard oracle}.


\begin{proposition}
Every locally Lipschitz function $f$ admits a cutting plane oracle. It admits a strict cutting plane oracle if and only if
it admits a strict model.
\end{proposition}

\begin{proof}
The first part is clear, because every $f$ has at least one model, $\phi^\sharp$, and every model $\phi$ generates an oracle $\mathscr O_\phi$.
Note that if $\phi$ is strict, then so is $\mathscr O_\phi$. Conversely, if $\mathscr O$ is strict, then its upper envelope (\ref{envelope})
is a strict model, as $(\widehat{\mathscr O}_3)$ is easily seen to ensure $(\widehat{M}_2)$.
\end{proof}


\subsection{Cutting planes from downshifted tangents}
\label{downshift}
One prominent way to generate cutting planes is the downshifted tangent oracle $\mathscr O^\downarrow$, which we now discuss.
Recall that for  $\g\in \partial f(\z)$
the affine function $t_{\z,\g}(\cdot)=f(\z) + \g^\ttop (\cdot - \z)$ is a tangent  to $f$ at $\z$, or simply a tangent plane.

\begin{definition}
{\rm
Let $\x$ be the current serious iterate, $\z$ a trial step. Let $t_{\z,\g}(\cdot)$ be a tangent of $f$ at $\z$.  
For a fixed constant $c>0$ we define the {\em downshifted tangent} $m_{\z,\g}^\downarrow(\cdot,\x)$ at serious iterate $\x$ and trial step $\z$
associated with the tangent $t_{\z,\g}(\cdot)$ as follows:
\begin{equation}
\label{down}
m_{\z,\g}^\downarrow(\cdot,\x) = 
t_{\z,\g}(\cdot) - \left[  t_{\z,\g}(\x) - f(\x) + c\|\z-\x\|^2\right]_+.
\end{equation}
The quantity appearing on the right
\begin{equation}
\label{downshift}
s(\z,\x,\g) = \left[ t_{\z,\g}(\x) - f(\x) + c\|\z-\x\|^2\right]_+
\end{equation}
 is called the {\em downshift} of the tangent.   The oracle is denoted
 $\mathscr O^\downarrow$.
}
 \hfill $\square$
\end{definition}

\begin{remark}
The explanation is that as long as the tangent  $t_{\z,\g}(\cdot)$ passes below $f(\x) - c \|\z-\x\|^2$ at $\x$, we can use it directly as cutting plane. However, if
$t_{\z,\g}(\x) > f(\x) - c\|\z-\x\|^2$, where $c > 0$ is the same fixed parameter, then we have to shift the tangent down to obtain the cutting plane
$m_{\z,\g}^\downarrow(\cdot,\x)$, 
so that now $m_{\z,\g}^\downarrow(\x,\x)= f(\x)-c\|\z-\x\|^2$, as otherwise the oracle would not respect axiom $(\mathscr O_1)$. The term $c\|\z-\x\|^2$ is chosen for convenience. Any
function $c(\cdot):\mathbb R \to \mathbb R_+$ with $c(t) = {\rm o}(t)$ as $t\to 0$ would give a similar oracle
with downshift $s=\left[t_{\z,\g}(\x) - f(\x) - c(\|\z-\x\|)\right]_+$.
\end{remark}

\begin{remark}
If $\z=\x$, then every tangent $t_{\x,\g}(\cdot)$ at $\x$ is also a cutting plane $m_{\x,\g}^\downarrow(\cdot,\x)$ at serious iterate $\x$ and trial step $\z=\x$, because $\g\in \partial f(\x)$
and because
the downshift (\ref{downshift}) at $z=x$ is automatically zero. 
\end{remark}

\begin{definition}
A cutting plane $m_{\x,\g}^\downarrow(\cdot,\x)$ at serious iterate $\x$ and trial step $\z=\x$  is called  an {\em exactness
plane}.  The same terminology is used for any other oracle $\mathscr O$. \hfill $\square$
\end{definition}

\subsection{Strictness of the downshift oracle $\mathscr O^\downarrow$}
\label{strictness_downshift}

Recall that a locally Lipschitz function $f$ is 
lower $C^1$ at $\x$ if for
$\x_j,\y_j\to \x$ and {\em every}  $\g_j \in \partial f(\y_j)$ there exist
$\epsilon_j \to 0^+$ such that $f(\y_j) \leq f(\x_j) + \g_j^\ttop(\y_j-\x_j) + \epsilon_j \|\y_j-\x_j\|$,  \cite{spingarn}.

\begin{proposition}
\label{mifflin}
The downshift operator $\mathscr O^\downarrow$ satisfies axioms $(\mathscr O_1)$, $(\mathscr O_2)$ and $(\mathscr O_4)$.
If $f$ is Clarke regular at $\x$, then $\mathscr O^\downarrow$  satisfies $(\mathscr O_3)$ at $\x$.
Suppose $f$ is in addition lower $C^1$ at $\x$. Then the downshift operator is  strict at $\x$, i.e., satisfied $(\widehat{\mathscr O}_3)$ at $\x$.
\end{proposition}

\begin{proof}
1) Axiom $(\mathscr O_1)$ is clear. We check 
$(\mathscr O_2)$. Let $(a_j,\g_j)\in \mathscr O^\downarrow(\z_j,\x)$ with $a_j\to f(\x)$,  $\g_j \to \g$, and $\z_j$ bounded. Passing to a subsequence,
assume $\z_j\to \z$. Now $m_j^\downarrow(\cdot,\x) = a_j + \g_j^\ttop(\cdot - \x) = f(\z_j) + \g_j^\ttop(\cdot - \z_j) - s_j$, where $s_j = s(\x,\z_j,\g)$ is the 
corresponding downshift (\ref{downshift}). From $a_j \to f(\x)$ we get
$f(\z_j) + \g_j^\ttop(\x-\z_j) - s_j \to f(\x) = f(\z) + \g^\ttop(\x-\z) - s$, where $s = \lim s_j$ and $\g \in \partial f(\z)$. There are two cases, either $s_j \to 0$,
or $s_j \to s > 0$. In the first case $s=0$ we obtain
\begin{equation}
\label{downshift2}
s_j = \left[    f(\z_j) + \g_j^\ttop(\x-\z_j) - f(\x) + c\|\z_j-\x\|^2\right]_+ \to 0.
\end{equation}
But since $f(\z_j) + \g_j^\ttop(\x-\z_j) - f(\x) \to 0$ due to $s=0$, we must also have $c\|\z_j-\x\|^2 \to 0$, proving $\z_j \to \x$. Hence $\z=\x$, and so
$\g \in \partial f(\x)$. Now in the case $s > 0$ we get $s_j > 0$ from some counter onward, hence
$f(\z_j) + \g_j^\ttop(\x-\z_j) -  f(\z_j) - \g_j^\ttop(\x-\z_j) + f(\x) - c\|\z_j-\x\|^2
\to f(\x) = f(\z) + \g^\ttop(\x-\z) - f(\z) - \g^\ttop(\x-\z) + f(\x) - c\|\z-\x\|^2$, which implies
$f(\x) + c\|\z_j-\x\|^2 \to f(\x)$, hence again $\z_j \to \x$, leading to the same conclusion.

2) Let us check  $(\mathscr O_4)$. This is obviously the analogue of the model axiom $(M_3)$.
Fix $\z_j \to \z$, $\y_j\to \y$ and $\x_j\to \x$,
and  $(a_j,\g_j)\in \mathscr O^\downarrow(\z_j,\x_j)$. Then $m_j^\downarrow(\cdot,\x_j) = a_j + \g_j^\ttop(\cdot - \x_j) = f(\z_j) + \g_j^\ttop(\cdot - \z_j) - s_j$ is the downshifted
tangent at $\z_j$, i.e.
\begin{equation}
\label{downshift3}
s_j = \left[    f(\z_j) + \g_j^\ttop(\x_j-\z_j) - f(\x) + c\|\z_j-\x_j\|^2\right]_+.
\end{equation}
Choose a subsequence $j\in J$ such that $\ell:=\displaystyle\lim_{j\in J} m_j^\downarrow(\y_j,\x_j) = \limsup_{j\to \infty} m_j^\downarrow(\y_j,\x_j)$.
Then this limit is $\ell= f(\z) + \g^\ttop(\y-\x) - s$, where $s=\lim_{j\in J} s_j$. We have to find $\z'$ such that $\ell \leq a' + \g'^\ttop(\y-\x)$ for $(a',\g')\in \mathscr O(\z',\x)$.
We simply choose $\z'=\z$ and $\g'=\g$, because then $f(\z) + \g^\ttop(\cdot-\z) - s$ should be the downshifted tangent at $\z$. 
For that to be confirmed, we have just to
show that $s \leq s(\z,\x,\g)$, in other words, we have to
show upper semi-continuity $\lim_{j\in J} s(\z_j,\x_j,\g_j) \leq s(\z,\x,\g)$ of the downshift. But that follows immediately here due to convergence
of all the elements in (\ref{downshift3}).

3)
We have to verify axiom $(\mathscr O_3)$ at $\x$ with $f$ regular at $\x$.   Consider a sequence $\y^k\to \x$. Let $m_k(\cdot,\x) := m_{\y^k,\g_k}^\downarrow(\cdot,\x)$ be a cutting plane at serious iterate $\x$
and trial point $\y^k$, where $\g_k\in \partial f(\y^k)$. We have to find $\epsilon_k\to 0^+$ such that
$f(\y^k) \leq m_k(\y^k,\x) + \epsilon_k \|\y^k-\x\|$.
Now let $t_{\y^k,\g_k}(\cdot)$ be the tangent to $f$ from which $m_k$ is downshifted, and let $s_k = s(\y^k,\x,\g_k)$ be the downshift. Consider the case 
$s_k > 0$. Then $m_k(\cdot,\x) = f(\x) + \g_k^\ttop (\cdot - \x) - c \|\y^k-\x\|^2$, and we have to find $\epsilon_k\to 0^+$  such that
\[
f(\y^k) \leq f(\x) + \g_k^\ttop (\y^k-\x) - c \|\y^k-\x\|^2 + \epsilon_k \|\y^k-\x\|.
\]
If we put $\h_k = (\y^k-\x)/\|\y^k-\x\|$, $t_k = \|\y^k-\x\|$, then this may be re-written as
\begin{equation}
\label{dini}
\frac{f(\x + t_k\h_k) - f(\x)}{t_k} \leq \g_k^\ttop \h_k + \epsilon_k.
\end{equation}
Passing to a subsequence, we may assume $\h_k \to \h$ for some $\|\h\|=1$, and then since $f$ is locally Lipschitz, the left hand term converges to
the Dini derivative $f'(\x,\h)$. On the right hand side, on the other hand,
we have $\limsup_{k\to \infty} \g_k^\ttop \h_k \leq f^\circ(\x,\h)$, and by regularity, the two limits coincide. 
In other words, we find that $\epsilon_k = \max\{0,t_k^{-1}(f(\x+t_k\h_k)-f(\x))-\g_k^\ttop \h_k\} \to 0^+$ does the job.
That proves $(\mathscr O_3)$.

4) It remains to verify $(\widehat{\mathscr O}_3)$ when $f$ is
lower $C^1$ at $\x$. Fix $\z_j,\x_j \to \x$ and let $(a_j,\g_j)\in \mathscr O^\downarrow(\z_j,\x_j)$. Then by hypothesis there exist
$\epsilon_j \to 0^+$ such that $f(\z_j) \leq f(\x_j) + \g_j^\ttop (\z_j-\x_j) + \epsilon_j \|\z_j-\x_j\|$. Adding $\g_j^\ttop(\cdot - \z_j)$ on both sides gives
$t_{\z_j,\g_j}(\cdot) \leq f(\x_j) + \g_j^\ttop (\cdot - \x_j) + \epsilon_j \|\z_j-\x_j\|$. Hence the downshift satisfies
$s_j = \left[ t_{\z_j,\g_j}(\x_j)-f(\x_j) + c\|\x_j-\z_j\|^2\right]_+
\leq  \epsilon_j \|\z_j-\x_j \| + c\|\x_j-\z_j\|^2 =:\widetilde{\epsilon}_j(\|\z_j-\x_j\|)$ with
$\widetilde{\epsilon}_j = \epsilon_j + c\|\z_j-\x_j\|\to 0^+$. 
Then $f(\z_j) = t_{\z_j,\g_j}(\z_j) = m_{\z_j,\g_j}^\downarrow(\z_j,\x_j) + s_j \leq m_{\z_j,\g_j}^\downarrow(\z_j,\x_j) + \widetilde{\epsilon}_j \|\z_j-\x_j\|$
shows strictness.
\end{proof}

A different way to say that $f$ is lower $C^1$ at $\bar{\x}$ is that for every $\epsilon > 0$ there exists $\delta > 0$
such that for all $\x,\z \in B(\bar{\x},\delta)$ and {\em every}  $\g\in \partial f(\z)$ one has
$
f(\x) - f(\z) \geq \g^\ttop (\x-\z) - \epsilon \|\x-\z\|,
$
which may also be written as
\[
f(\x) - f(\z) \geq f^\circ(\z,\x-\z) - \epsilon \|\x-\z\|.
\]
We weaken  this as follows:  We say that $f$ is weakly lower $C^1$
at $\bar{\x}$ if for every $\epsilon > 0$ there exists $\delta > 0$
such that for all $\x,\z\in B(\bar{\x},\delta)$ and {\em some} $\g \in \partial f(\z)$ we have
$f(\x)-f(\z) \geq \g^\ttop (\x-\z) - \epsilon \|\x-\z\|$, which may be written as 
\begin{equation}
\label{wlc1}
f(\x)-f(\z) \geq -f^\circ(\z,\z-\x) - \epsilon \|\x-\z\|.
\end{equation}
There is an oracle associated with this property, which is slightly stronger than
$\mathscr O^\downarrow$, and which we denote $\mathscr O^{\downarrow\downarrow}$. For $\mathscr O^{\downarrow\downarrow}$ we
do not take an arbitrary tangent to $f$ at $\z$,
but the specific one $t_{\z,\g'}(\cdot)$
where $\g'\in \partial f(\z)$ satisfies $\g'^\ttop(\x-\z) = -f^\circ(\z,\z-\x)$. Then we downshift $t_{\z,\g'}$ as before to obtain the $\mathscr O^{\downarrow\downarrow}$-cutting plane.

\begin{proposition}
Suppose  $f$ is weakly lower $C^1$ at $\x$. Then $\mathscr O^{\downarrow\downarrow}$ is strict at $\x$. In consequence, every weakly lower $C^1$
function has a strict model.
\end{proposition}

\begin{proof}
We have to check $(\widehat{\mathscr O}_3)$. Let $\z_j,\x_j\to \x$, then $(a_j,\g_j')= \mathscr O^{\downarrow\downarrow}(\z_j,\x_j)$ satisfies
$\g_j'^\ttop(\x_j-\z_j)=\inf\{\g^\ttop (\x_j-\z_j): \g\in \partial f(\z_j)\}$. By (\ref{wlc1}) we have $f(\x_j)-f(\z_j) \geq \g_j'^\ttop (\x_j-\z_j) - \epsilon_j \|\x_j-\z_j\|$ for
certain $\epsilon_j\to 0^+$. That can be written $f(\x_j) \geq t_{\z_j,\g_j'}(\x_j) - \epsilon_j \|\z_j-\x_j\|$, hence
$\epsilon_j \|\z_j-\x_j\| + c \| \z_j-\x_j\|^2 \geq  t_{\z_j\g_j'}(\x_j)-f(\x_j) + c\|\z_j-\x_j\|^2$. This shows that the downshift
(\ref{downshift}) is $s_j \leq \widetilde{\epsilon}_j \|\z_j-\x_j\|$, where $\widetilde{\epsilon}_j = \epsilon_j + c\|\z_j-\x_j\| \to 0^+$. 
Hence with the argument of part 4) above, $\mathscr O^{\downarrow\downarrow}$ is strict at ${\x}$.
\end{proof}

\begin{proposition}
Suppose the function  $f$ is weakly upper $C^1$ or weakly lower $C^1$ at every $\x$. Then $f$ admits
a strict cutting plane oracle, and consequently has a strict model.
\end{proposition}

\begin{proof}
We can use the idea of remark \ref{idea}. The standard oracle $\mathscr O^\sharp$ is strict at those $\x$ where $f$ is weakly upper $C^1$,
and  $\mathscr O^{\downarrow\downarrow}$ is strict at those $\x$ where $f$ is weakly lower $C^1$, so the oracle
$\mathscr O^{\sharp\downarrow\downarrow}(\z,\x)$ which among the two possible planes takes the one which has the larger value at $\z$ will
be strict. Naturally, one could also take the oracle $\mathscr O^\sharp \vee \mathscr O^{\downarrow\downarrow}$ which takes both planes.
The corresponding envelope models are then strict. 
\end{proof}

\begin{example}
For  $f$ in example \ref{zig_zag} neither $\mathscr O^\downarrow$ nor
$\mathscr O^{\downarrow\downarrow}$ is strict at $0$. With the same notation,
for $s_k \in (t_{2k},t_{2k-1})$ the slope of the tangent at $s_k$ is $-1$, so the cutting plane
$\mathscr O^\downarrow(s_k,0)$
is the line with slope $-1$ passing through the point $(0,-cs_k^2)$, i.e.,  $m^\downarrow(s_k,0) = -cs_k^2 -s_k$. For $(\mathscr O_3)$ at $0$ we would require
$f(s_k) \leq m^\downarrow(s_k,0) + {\rm o}(s_k)$, hence $t_{2k}^2 - (s_k-t_{2k}) \leq -cs_k^2-s_k + {\rm o}(s_k)=-s_k+{\rm o}(s_k)$. Choose $s_k$ such that $s_k-t_{2k} = t_{2k}^2/2$, then 
$s_k = t_{2k} + t_{2k}^2/2$, so ${\rm o}(s_k)={\rm o}(t_{2k})$, and we have to assure
$t_{2k}^2/2 \leq -t_{2k} - t_{2k}^2/2 + {\rm o}(t_{2k})$, and that is impossible, as the right hand side is asymptotically $<0$. Here the downshift
operator $\mathscr O^\downarrow$ is not even an oracle at $0$, let alone a strict one.
 A similar argument applies to $f(\x) = \x^2 \sin\x^{-1}$.
 \end{example}

\begin{remark}
Downshift is used in the early versions of the bundle technique to deal heuristically with non-convex cases.
See e.g. the codes N1CV2 and N2BN1 by Lemar\'echal and Sagastiz\'abal \cite{lema-saga}, or in 
Zowe's BT package \cite{zowe}.  
Mifflin \cite{mifflin3}  justifies the downshift oracle theoretically for a linesearch algorithm. For our justification of downshift in the bundle method see e.g. \cite{cutting,gabarrou}.
\end{remark}

\subsection{Oracles with infinitely many cuts}
\label{many}
We consider an eigenvalue optimization problem
$$\min\{ \lambda_1 \left( F(\x)\right): \x \in C\},$$ where $F:\mathbb R^n \to \mathbb S^m$ is a $C^1$-mapping  into the space 
$\mathbb S^m$ of $m\times m$ symmetric or Hermitian matrices, and $\lambda_1:\mathbb S^m\to \mathbb R$ is the maximum eigenvalue. Suppose we use the natural model
$\phi(\cdot,\x) = \lambda_1 \left( F(\x) + F'(\x)(\cdot-\x) \right)$ of $f=\lambda_1 \circ F$. Let
$\mathscr G = \{G \in \mathbb S^m: G \succeq 0, {\rm Tr}(G)=1\}$, then
$\lambda_1(X) = \max\{G \bullet X: G \in \mathcal G\}$, where $G\bullet X = {\rm Tr}(GX)$ is the scalar product in $\mathbb S^m$.
Let
$\z^k$ be a null step and  suppose the multiplicity of $\phi(\z^k,\x) = \lambda_1\left( F(\x) + F'(\x) (\z^k-\x) \right)$ is $r > 1$. Let
$Q_k$ be a $m \times r$-matrix whose $r$ columns form an orthonormal basis of the maximum eigenspace of $F(\x) + F'(\x)(\z^k-\x)$.
Put $\mathcal G_k =\{Q_k^\ttop YQ_k: Y \in \mathbb S^r, Y \succeq 0, {\rm Tr}(Y)=1\}$, then $\mathcal G_k \subset \mathcal G$
and 
\begin{align*}
&\max\left\{G'\bullet \left[F(\x) + F'(\x)(\cdot - \x)\right]: G' \in \mathcal G_k\right\} \\
&\quad=\max\left\{ Y \bullet \left[Q_k ( F(\x) + F'(\x) (\cdot-\x) \right] Q_k^\ttop: Y\succeq 0, {\rm Tr}(Y)=1\right\}\\
&\quad= \lambda_1 \left( Q_k \left[  F(\x) + F'(\x) (\cdot-\x) \right] Q_k^\ttop  \right). 
\end{align*}
This means 
we choose as oracle the infinite set $\mathscr O^{\rm spec}(\z^k,\x)=\{(a(Y),\g(Y)): Y \succeq 0, {\rm Tr}(Y)=1\}$, where
$a(Y)=Y\bullet Q_k F(\x) Q_k^\ttop$ and $\g(Y)=F'(\x)^*  Q_k^\ttop YQ_k$, with $Q_k^\ttop YQ_k \in \mathbb S^m$ and
$F'(\x)^*:\mathbb S^m\to \mathbb R^n$ the adjoint of $F'(\x):\mathbb R^n \to \mathbb S^m$.
For practical aspects of this type of oracle, which leads to spectral bundle methods, see \cite{ANP:2008,ANP:2009,helm1,helm2,helm3,LO:00,na:03_1,na:03_2,thevenet:04,TNA:04,TNA:05}.
The key observation is that the tangent program (\ref{tangent}) in this approach will be a convex SDP.

\subsection{Working models}
In our trust-region method the tangent program is based on
a {\em working model}  $\phi_k(\cdot,\x)$ of $f$ at serious iterate $\x$.
In the bundle literature originating from Lemar\'echal's \cite{lemarechal1,lemarechal2}, this model
is traditionally denoted as $\check{f}_k$ and referred to as a model of $f$ at $\x$. In our approach we distinguish between model $\phi$ and working model $\phi_k$
from a reason which will become clear shortly.

Let $\mathscr O$ be a cutting plane oracle for $f$. Then the working model at serious iterate $\x$ and inner loop counter $k$ has the form
$\phi_k(\cdot,\x)=\sup\{ a + \g^\ttop(\cdot - \x): (a,\g)\in \mathscr W_k\}$, where
the sets $\mathscr W_k$ are generated recursively through $\mathscr O$: Suppose $\z^k$ is an unsuccessful trial step (a null step)
obtained via (\ref{trial}) from the solution $\y^k$ of the tangent program (\ref{tangent}) based on the $k^{\rm th}$ working model $\phi_k(\cdot,\x)$. Then the $(k+1)^{\rm st}$ working model
$\phi_{k+1}(\cdot,\x) = \sup\{ a + \g^\ttop(\cdot - \x): (a,\g) \in \mathscr W_{k+1}\}$ is obtained from $\phi_k$ by the following rules:
\begin{itemize}
\item[$(W_1)$] At least one exactness plane $f(\x) + \g^\ttop(\cdot-\x)$ with  $(f(\x),\g)\in \mathscr O(\x,\x)$ is included in $\phi_{k+1}$. That is,
$(f(\x),\g) \in \mathscr W_{k+1}$ for some  $(f(\x),\g)\in \mathscr O(\x,\x)$.
\item[$(W_2)$] The aggregate plane $a^*_k + \g^{*\ttop}_k(\cdot-\x)$ associated with the solution $\y^k$ of the $k^{\rm th}$-tangent program (\ref{tangent})
is included in $\phi_{k+1}$, i.e., $(a^*_k,\g^*_k)\in \mathscr W_{k+1}$. (See section \ref{tangent}).
\item[$(W_3)$] All cutting planes $(a,\g)\in \mathscr O(\z^k,\x)$ are simultaneously included in $\phi_{k+1}$, that is,
$\mathscr O(\z^k,\x) \subset \mathscr W_{k+1}$.
\item[$(W_4)$] Planes contributing to $\phi_{k+1}(\cdot,\x)$ via $\mathscr W_{k+1}$ other than those  in $(W_1) - (W_3)$ must already have been elements of 
$\mathscr W_k$, but not all $(a,\g)\in \mathscr W_k$ are kept in
$\mathscr W_{k+1}$.
\end{itemize}

The initialization of $\phi_1(\cdot,\x)$ is as follows. We request
that $\phi_1(\cdot,\x)$  contains at least one exactness plane generated by $\mathscr O$, i.e., $f(\x) + \g^\ttop(\cdot - \x) \leq \phi_1(\cdot,\x)$ for some
$(f(\x),\g)\in \mathscr O(\x,\x)$. In other words, $(f(\x),\g) \in \mathscr W_1 \subset \mathscr O(\x,\x)$.

If in addition a positive semi-definite symmetric matrix $Q(\x)\succeq 0$ is available as a substitute for the Hessian of $f$ at $\x$, then
we call $\Phi_k(\cdot,\x) = \phi_k(\cdot,\x) + \frac{1}{2} (\cdot-\x)^\ttop Q(\x) (\cdot-\x)$
a {\em second-order working model} of $f$ at serious iterate $\x$.  We shall occasionally
use the semi-norm $|\y|_Q^2 = \y^\ttop Q\y$, so that $\Phi_k(\cdot,\x) = \phi_k(\cdot,\x) + \frac{1}{2}|\cdot - \x|_Q^2$.


\begin{remark}
By construction we have $\phi_k(\cdot,\x) \leq \phi(\cdot,\x)$ at all inner loop instants $k$, where $\phi$ is the envelope model (\ref{envelope}) of $\mathscr O$.
By convexity we automatically have $\partial_1 \phi_k(\x,\x) \subset \partial_1 \phi(\x,\x) \subset \partial f(\x)$, so that
$0\in \partial ( \phi_k(\cdot,\x) +i_C)(\x)$ implies that $\x$ is a Clarke critical point of (\ref{program}).   This is crucial for practice, as we can stop the algorithm as soon as the tangent program
based on the working model $\phi_k(\cdot,\x)$
finds no model reduction step.
\end{remark}

\begin{remark}
Note that we do not necessarily have $\mathscr W_k \subset \mathscr W_{k+1}$, as this would lead to
tangent programs of increasing size. Rules $(W_1)-(W_4)$ only require that:
a) An exactness plane  assures  $\phi_{k+1}(\x,\x) = f(\x)$; b) the aggregate plane $(a_k^*,\g_k^*)$  assures $\phi_{k+1}(\y^k,\x) \geq \phi_k(\y^k,\x)$,
and c) the cutting planes $(a_k,\g_k)\in \mathscr O(\z^k,\x)$ are intended to bring the value $\phi_{k+1}(\z^k,\x)$ as close as possible
to the unduly large value $f(\z^k)$, because that value was much larger than the predicted value $\Phi_k(\z^k,\x)$, having caused the failure of $\z^k$.
\end{remark}

\begin{example}
Consider  $f(\x) = \max_{i\in I} f_i(\x)$, where $I$ is infinite and the $f_i$ are class $C^1$. The natural model is $\phi(\cdot,\x)=\max_{i\in I} f_i(\x) + \nabla f_i(\x)^\ttop(\cdot-\x)$,
but computing it might be costly for large size  $I$, and we might prefer a working
model  $\phi_k(\cdot,\x) = \max_{i \in I_k} f_i(\x) + \nabla f_i(\x)^\ttop (\cdot - \x)$ with some small subset $I_k$ of $I$. 
In $H_\infty$-optimization \cite{tac,multidisk,ANP:2008,ANP:2009} we have observed that diligent choices of an initial set $I_1 \subset I$ may greatly influence the performance of the method,
see also \cite{boyd1,boyd2}. This is encouraged by rule $(W_4)$.
\end{example}


\section{Non-smooth trust-region method}
\label{optimize}
In this section we present the main elements of the algorithm.

\subsection{Tangent program, aggregate plane and trial steps}
\label{tangent}
Once a second-order working model $\Phi_k(\cdot,\x)=\phi_k(\cdot,\x) + \frac{1}{2}(\cdot-\x)^\ttop Q(\x)(\cdot-\x)$ at serious iterate $\x$ and inner loop counter $k$ is fixed,
we solve the trust-region tangent program
\begin{eqnarray}
\label{tangent}
\begin{array}{ll}
\mbox{minimize} & \Phi_k(\y,\x) \\
\mbox{subject to} &\y\in C\\
& \|\y - \x\| \leq R_k, 
\end{array}
\end{eqnarray}
where $R_k > 0$ is the current trust-region radius. We obtain a solution $\y^k\in C$ of (\ref{tangent}), 
which  is unique in case $Q(\x) \succ 0$. For a polyhedral norm
like $\|\cdot\| = |\cdot|_\infty$ or $\|\cdot \|=|\cdot|_1$ and a polyhedron $C$, program (\ref{tangent})
reduces to a convex quadratic program, as long as the first-order working models are
polyhedral, or even an LP if $Q(\x)=0$. In the spectral bundle method \cite{ANP:2008,ANP:2009}, where $\phi_k$ contains an infinity of cuts,  (see section \ref{many}), the tangent program
is a convex SDP, and $C$ may then be allowed to be an SDP-constrained set. 
In this case the choice $\|\cdot\| = |\cdot|_2$ is acceptable, as the trust-region can be turned into a conical constraint, so that the tangent program is a convex SDP.

Note that by the necessary optimality conditions
for the tangent program in step 4 of the algorithm, there exists a subgradient 
$\g_k^*\in \partial \left( \phi_k(\cdot,\x)+i_C\right)(\y^k)$ such that $\g_k^*+Q(\y^k-\x)+\vv_k=0$, where $\vv_k$ is in the normal cone
to the trust-region norm ball $B(\x,R_k)$ at $\y^k$, and where $i_C$ is the indicator function of the convex set $C$.

\begin{definition}
\label{aggregate}
{\rm
We call $\g_k^*$ the {\em aggregate subgradient}.
The affine function
$m_k^*(\cdot,\x) = \phi_k(\y^k,\x) + \g_k^{*\ttop} (\cdot - \y^k)$ is called the {\em aggregate plane}.
}
 \hfill $\square$
\end{definition}

Now along with the solution $y^k$  of the tangent program 
we consider a larger set of admissible trial steps $z^k$. Fixing constants $\Theta \geq 1$ and $0 < \theta \ll 1$, we admit as trial point
any $\z^k\in C$ satisfying
\begin{equation}
\label{trial}
f(\x) - \Phi_k(\z^k,\x) \geq \theta \left( f(\x) - \Phi_k(\y^k,\x)  \right)  \mbox{ and }
\|\z^k - \x\| \leq \Theta \|\y^k-\x\|.
\end{equation}

\subsection{Acceptance and building the new working model}
Once a trial point $\z^k$ associated with a solution $\y^k$ of the tangent program
(\ref{tangent}) has been determined as in (\ref{trial}), acceptance is tested by the standard
test
\begin{equation}
\label{gamma}
\rho_k = \frac{f(\x) - f(\z^k)}{f(\x) - \Phi_k(\z^k,\x)} \stackrel{?}{\geq} \gamma
\end{equation}
for fixed $0 < \gamma \ll 1$. If $\rho_k <\gamma$ the trial point $\z^k$ is  rejected (a null step).  Then
we have to keep the inner loop going. Applying rules $(W_1)-(W_4)$ we generate  cutting planes $(a,\g) \in \mathscr O(\z^k,\x)$ at $\z^k$, 
the aggregate plane $(a^*,\g^*)$ at $\y^k$, and add those into the new $\phi_{k+1}(\cdot,\x)$, tapering out the old $\phi_k(\cdot,\x)$ by dropping some of the older cuts
contributing to the aggregate plane. We also keep at least one exactness plane $(a_0,\g_0)\in \mathscr O(\x,\x)$
in  the model $\phi_{k+1}(\cdot,\x)$, which may, or may not be, the one used at counter $k$. 



\subsection{Management of the trust-region radius}
Understanding
step 7 of the algorithm is crucial, because here a main difference with the classical trust-region  method occurs.
Namely, in the case of an unsuccessful trial step $\z^k$ in the inner loop at $\x$ we do {\em not} automatically reduce the trust-region  radius. 
We first call the oracle to provide cutting planes at $\z^k$. Then the secondary test 
in step 7 
\begin{equation}
\label{secondary}
\widetilde{\rho}_k = \frac{f(\x) - \phi_{k+1}(\z^k,\x)}{f(\x) - \Phi_k(\z^k,\x)}  \stackrel{?}{\geq} \widetilde{\gamma}
\end{equation}
(with $0 < \gamma < \widetilde{\gamma} < 1$)
serves to decide whether or not to reduce $R_k$ at the next sweep $k+1$ of the inner loop.  The rationale of this test is as follows:
If $\widetilde{\rho}_k < \widetilde{\gamma}$, then the effect of the cutting plane(s) at $\z^k$ is sensible, i.e., it is reasonable to believe that 
we could have performed better, had we already included this cutting plane in the working model $\phi_k$. 
In that case we keep enriching the working model by cuts and maintain $R_k$ unchanged, being reluctant to reduce $R_k$ prematurely, as this leads to 
unnecessarily small steps.
In the opposite case $\widetilde{\rho}_k \geq \widetilde{\gamma}$ the new cutting plane(s) do(es) not seem
to contribute anything substantial at $\z^k$, and here we reduce the trust-region radius in order to get closer to the current $\x$,
where progress is ultimately possible (due to $0 \not\in \partial f(\x)+N_C(\x))$. 

\begin{remark}
The secondary test appears for the first time in the first non-convex version of the bundle method
\cite{NPR:08}, see also \cite{ANP:2008,ANP:2009}. It can also be used with $\Phi_{k+1}(\z^k,\x)$ instead of $\phi_{k+1}(\z^k,\x)$ in the numerator.
\end{remark}

\subsection{Algorithm}
We are now ready to present the bundle trust-region algorithm. (See algorithm \ref{algo1}).
{\small
\begin{algorithm}[ht!]
\caption{Non-smooth trust-region algorithm \label{algo1}}
\begin{algo}
\PARAMETERS $0 < \gamma < \widetilde{\gamma} < 1$, $0 < \gamma < \Gamma \leq 1$, $0 < \theta \ll 1$, $\Theta \geq 1$, $q > 0$.
\STEP{Initialize outer loop} Fix initial iterate $\x^1\in C$ and memory trust-region radius $R_1^\sharp > 0$. Initialize $Q_1 \succeq 0$ with $\|Q_1\|\leq q$.
Put outer loop counter $j=1$.
\STEP{Stopping test} At outer loop counter $j$, stop if $\x^j$ is a critical point of (\ref{program}). Otherwise go to inner loop.
\STEP{Initialize inner loop}  Put inner loop counter $k=1$ and initialize trust-region radius as $R_1 = R_j^\sharp$. Build  first-order working model 
$\phi_1(\cdot,\x^j)$, including at least one exactness plane at $\x^j$. Possibly add recycled planes from previous steps. Build second-order working model
$\Phi_1(\cdot,\x^j) = \phi_1(\cdot,\x^j) + \frac{1}{2}(\cdot-\x^j)^\ttop Q_j (\cdot-\x^j)$.
\STEP{Trial step generation} At inner loop counter $k$ compute solution $\y^k$ of trust-region tangent program \\
\vspace*{.1cm}
\hspace*{4cm}$\mbox{minimize} \quad\; \Phi_k(\y,\x^j)$ \\
\hspace*{4cm}$\mbox{subject to}  \;\; \; \|\y-\x^j\| \leq R_k,\y\in C$\\
\vspace*{.1cm}
Admit any  $\z^k\in C$ satisfying $\|\z^k-\x^j\| \leq \Theta \| \y^k-\x^j\|$ and $f(\x^j)-\Phi_k(\z^k,\x^j) \geq
\theta \left( f(\x^j) - \Phi_k(\y^k,\x^j) \right)$ as trial step. 
\STEP{Acceptance test} If
$$\rho_k =\displaystyle \frac{f(\x^j)-f(\z^k)}{f(\x^j)-\Phi_k(\z^k,\x^j)} \geq \gamma$$ 
put $\x^{j+1}=\z^k$ (serious step), quit inner loop and goto step 8. Otherwise (null step), continue inner loop with step 6.
\STEP{Update working model}
Use aggregation to taper out model $\phi_k(\cdot,\x^j)$. Then call oracle $\mathscr O$ to
generate cutting planes $m_k(\cdot,\x^j)$ at $\z^k$. Include aggregate plane and cutting planes in new first-order working model $\phi_{k+1}(\cdot,\x^j)$.
Generate second-order working model $\Phi_{k+1}(\cdot,\x^j) = \phi_{k+1}(\cdot,\x^j) + \frac{1}{2} (\cdot-\x^j)^\ttop Q_j (\cdot-\x^j)$. 
\STEP{Update trust-region radius} Compute secondary control parameter \\
\hspace*{3cm}$$\widetilde{\rho}_k =\displaystyle \frac{f(\x^j)-\phi_{k+1}(\z^k,\x^j)}{f(\x^j)-\Phi_{k}(\z^k,\x^j)}   $$
and put 
$$R_{k+1}=\bigg\{ {R_k \; \, \,\,\mbox{ if }\; \widetilde{\rho}_k < \widetilde{\gamma} \atop \frac{1}{2}R_k \; \mbox{ if } \;\widetilde{\rho}_k \geq \widetilde{\gamma}}$$
Increase inner loop counter $k$ and go back to step 4.
\STEP{Update memory radius} Store new memory trust-region radius
$$ R_{j+1}^\sharp = \bigg\{ { R_k\, \,\,\; \mbox{ if }\; \rho_k < \Gamma \atop 2 R_k \; \mbox{ if }\; \rho_k \geq \Gamma}$$
Update $Q_j \to Q_{j+1}$ respecting $Q_{j+1} \succeq 0$ and $\|Q_{j+1}\| \leq q$. Increase outer loop counter $j$ and go back to step 2.
\end{algo}
\end{algorithm}
}

\newpage

\subsection{Recycling of planes}
When a new serious iterate $\x^+$ is found, then in the next sweep in step 3 
a new initial first-order working model $\phi_1(\cdot,\x^+)$ has to be generated. Contrary to the convex case,
we cannot keep cuts generated around $\x$ to stay in the model for $\x^+$.  At least in the case of the downshift oracle $\mathscr O^\downarrow$ 
we can recycle them as follows: If $(a,\g)\in \mathscr O^\downarrow(\z,\x)$ is in the model $\phi_k(\cdot,\x)$ at acceptance,
then include $(a^+,\g) \in \mathscr O^\downarrow(\z,\x^+)$ in $\phi_1(\cdot,\x^+)$ where $a^+=f(\z)+\g^\ttop(\x^+-\z) - s(\z,\x^+,\g)$.

\subsection{Practical aspects}
Consider the case where the first-order working model
has the form $\phi_k(\cdot,\x^j) = \max_{i\in I_k} a_i + \g_i^\ttop (\cdot - \x^j)$ for some finite set $I_k$. Suppose the
trust-region norm is the maximum norm, and $C = \{\x: A\x \leq b\}$ is a polyhedron. Then the tangent  program  
is the following CQP
\begin{eqnarray}
\label{QPtangent}
\begin{array}{ll}
\mbox{minimize} & t + \frac{1}{2} (\y-\x^j)^\ttop Q_j(\y-\x^j) \\
\mbox{subject to} & a_i + \g_i^\ttop (\y-\x^j) \leq t, \;\; i\in I_k\\
&A\y \leq b \\
&-R_k \leq \y_i - \x_i^j \leq R_k \;\;, i=1,\dots,n
\end{array}
\end{eqnarray}
with decision variable $(t,\y) \in \mathbb R\times \mathbb R^n$, giving rise to the solution $\y^k$ in step 4. If $Q_j=0$, then
the tangent program is even a LP.

\begin{remark}
Axioms $(W_1)$, $(W_4)$ can be further relaxed by allowing so-called {\em anticipated cutting planes}. That means we can call the oracle $\mathscr O(z,x)$  at
points $z$ other than $y^k,z^k$ and include those cutting planes into the working models.   This still gives convergence  and allows to exploit the specific structure of a
problem to speed up acceptance in the inner loop.
\end{remark}

\section{Convergence}
\label{convergence}
In this section we prove  convergence of the trust-region  algorithm toward a critical point $0 \in \partial f(\x^*)+N_C(\x^*)$. We prepare with three technical lemmas in section \ref{sect_est} and
prove termination of the inner loop in section \ref{finiteness}.
During this part of the proof we write $\x = \x^j$ and $Q = Q_j$, as those elements are fixed during the inner loop at 
counter $j$. We write $|\cdot|_Q$ for the seminorm $|\x|_Q^2 = \x^\ttop Q\x$.

\subsection{Three technical lemmas}
\label{sect_est}

\begin{lemma}
\label{est}
There exists  $\sigma > 0$, depending only on the 
trust-region norm $\|\cdot\|$, such that 
the solution $\y^k$ of the trust-region tangent program in step {\rm 4}, with the corresponding aggregate subgradient $\g^*_k \in \partial \left(\phi_k(\cdot,\x)+i_C\right)(\y^k)$, 
satisfies the estimate
\begin{equation}
\label{kolmogoroff1}
f(\x)-\phi_k(\y^k,\x) \geq \sigma \|\g^*_k + Q(\y^k-\x)\| \| \y^k-\x\|.
\end{equation}
\end{lemma}

\begin{proof}
1)
Let $\|\cdot\|$ be the norm used in the trust-region tangent program, $|\cdot|$ the standard Euclidian norm.
There exists $\epsilon >0$ such that
$|\uu|\leq \epsilon$ implies $\|\uu\|\leq 1$. Now if $\|\uu\|=1$ and if $\vv$ is in the normal cone to
the $\|\cdot\|$-unit ball at $\uu$, we have $\vv^\ttop (\uu-\uu') \geq 0$ for every $\|\uu'\|\leq 1$ by the normal cone criterion. 
Hence $\vv^{\ttop}(\uu-\uu')\geq 0$ for every $|\uu'|\leq \epsilon$ by the above, and using $\uu'=\epsilon \vv  / |\vv|$
that implies $\vv^\ttop \uu \geq \epsilon |\vv|$.

2)
Since $\y^k$ is an optimal solution of (\ref{tangent}), we have
$0=\g_k^*+ Q(\y^k-\x) + \vv_k$ by the optimality condition, 
where $\g_k^*\in \partial\left( \phi_k(\cdot,\x)+i_C \right)(\y^k)$ is the aggregate subgradient  and  $\vv_k$ a normal vector to
the $\|\cdot\|$-norm ball  $B(\x,R_k)$ at $\y^k$. 
By the subgradient inequality, using $\x,\y^k\in C$, we have
\[
\g_k^{*\ttop} (\x-\y^k) \leq \phi_k(\x,\x)-\phi_k(\y^k,\x) = f(\x)-\phi_k(\y^k,\x).
\]
Now by part 1), on putting ${\uu}_k = (\y^k-\x)/\|\y^k-\x\|$, we have
$\vv_k^\ttop \uu_k \geq \epsilon |\vv_k|$ independently of  $k$,
because $\vv_k$, being normal to the $\|\cdot\|$-ball of radius $\|\y^k-\x\|$ and center $0$ at $\y^k-\x$, 
is also normal to the $\|\cdot\|$-unit ball at $\uu_k$. 
 But then
$\g_k^{*\ttop} (\x-\y^k)=\vv_k^\ttop (\y^k-\x) + (\x-\y^k)^\ttop Q(\x-\y^k) \geq 
\vv_k^\ttop(\y^k-\x) \geq
\epsilon |\vv_k| \|\y^k-\x\| \geq \epsilon^2 \|\vv_k\| \|\y^k-\x\|
=\epsilon^2 \|\g_k^*+Q(\y^k-\x)\|\|\y^k-\x\|$, where we have used part 1). 
This proves the result with $\sigma = \epsilon^2$.
\end{proof}

\begin{lemma}
There exists constants $\sigma' >0, \sigma''\geq 0$, depending only on the trust region norm $\|\cdot\|$ and the parameters 
$\theta$, $\Theta$, $q$ used in the algorithm, such that
the trial points $\z^k$ in step {\rm 4} of the algorithm, associated with the solutions $\y^k$ of the tangent program {\rm (\ref{tangent})} 
and aggregate subgradients $\g_k^*\in \partial \left( \phi_k(\cdot,\x) + i_C  \right)(\y^k)$,
satisfy the estimate
\begin{equation}
\label{11}
f(\x) - \Phi_k(\z^k,\x) \geq \left(\sigma' \|\g_k^*\| - \sigma'' \|\y^k-\x\| \right) \|\y^k - \x\|.
\end{equation}
\end{lemma}

\begin{proof}
Subtracting $\frac{1}{2}|\y^k-\x|_Q^2$ from both sides of (\ref{kolmogoroff1}) and using
$|\y^k-\x|_Q^2 \leq \|Q\|\|\y^k-\x\|^2$ gives
\[
f(\x) - \Phi_k(\y^k,\x) \geq  \left(  \epsilon^2 \|\g_k^*+Q(\y^k-\x)\| -\textstyle \frac{1}{2}\|Q\| \|\y^k -\x\| \right) \|\y^k-\x\|.
\]
Hence by (\ref{trial}), the triangle inequality, and $\|Q\|\leq q$ in step 8 of the algorithm, we have
\begin{align}
\label{12}
f(\x) - \Phi_k(\z^k,\x)& \geq  \theta \left(  \epsilon^2 \|\g_k^*+Q(\y^k-\x)\| -\textstyle \frac{1}{2}\|Q\| \|\y^k -\x\| \right) \|\y^k-\x\|   \notag \\ \notag
&\geq \theta \epsilon^2 \|\g_k^*\| \|\y^k-\x\| - (\textstyle\frac{1}{2}-\theta\epsilon^2)\|Q\|  \|\y^k-\x\|^2\\
&\geq \left( \theta \epsilon^2 \|\g_k^*\|  - (\textstyle\frac{1}{2}-\theta\epsilon^2) q  \|\y^k-\x\| \right)\|\y^k-\x\|. \end{align}
   This is (\ref{11}) with $\sigma'=\theta\epsilon^2$ and $\sigma''= \max\{0,(\frac{1}{2}-\theta\epsilon^2)q\}$. 
\end{proof}

\begin{lemma}
\label{est3}
Suppose $\Delta_k=\|\y^k-\x\| / \|\g_k^*\| \to 0$ as $k \to \infty$. Then there exists a constant $\sigma > 0$, depending only
on $\theta,\Theta,q$ and the trust region norm $\|\cdot\|$, such that from some counter $k_0$ onward,
\begin{equation}
\label{kolmogoroff}
f(\x) - \Phi_k(\z^k,\x) \geq \sigma \|\g_k^*\| \|\z^k-\x\|.
\end{equation}
The counter $k_0$ can be chosen smallest with the property that
$\Delta_k < \frac{1}{2}\theta\epsilon^2/(\frac{1}{2}-\theta\epsilon^2)q$ for all $k \geq k_0$.
\end{lemma}

\begin{proof}
The estimate is obvious from the previous Lemma if $\frac{1}{2} \leq \theta \epsilon^2$, as then $\sigma''=0$ in (\ref{11}). Otherwise, since
$\|\y^k-\x\|/\|\g_k^*\| \to 0$, there exists an index $k_0$ such that
\begin{equation}
\label{13}\Delta_k=\frac{\|\y^k-\x\|}{ \|\g_k^*\|}  < \frac{\frac{1}{2}\theta\epsilon^2 }{\left( \frac{1}{2}-\theta\epsilon^2  \right)q}\end{equation}
for all $k \geq k_0$.
Then from (\ref{12})
\begin{align*}
f(\x) - \Phi_k(\z^k,\x) &\geq  \left( \theta \epsilon^2 \|\g_k^*\|  - (\textstyle\frac{1}{2}-\theta\epsilon^2) q  \|\y^k-\x\| \right)\|\y^k-\x\| \\
&\geq \textstyle\frac{1}{2} \theta\epsilon^2 \|\g_k^*\| \|\y^k-\x\| \qquad\qquad\qquad \mbox{  (using (\ref{13}))}\\
&\geq  \textstyle\frac{1}{2}\theta \epsilon^2 \Theta^{-1} \|\g_k^*\| \|\z^k-\x\|  \quad\qquad\quad\, \mbox{  (using (\ref{trial}))}
\end{align*}
for all $k \geq k_0$. This is the requested estimate with $\sigma = \frac{1}{2}\theta\epsilon^2\Theta^{-1}$.
\end{proof}

\subsection{Finite termination of the inner loop}
\label{finiteness}
We now investigate whether the inner loop at the current serious iterate $\x$ can find a new
serious iterate $\x^+$ satisfying the acceptance condition (\ref{gamma}), or whether it fails and turns infinitely.
The hypothesis for the inner loop is that the cutting plane oracle satisfies $(\mathscr O_1)-(\mathscr O_4)$. Strictness $(\widehat{\mathscr O}_3)$  of $\mathscr O$
is not needed for termination of the inner loop.

\begin{lemma}
\label{short}
Suppose the inner loop at serious iterate $\x$ turns infinitely with $\displaystyle\liminf_{k\to \infty} R_k=0$. Then $\x$ is a critical point of {\rm (\ref{program})},
i.e., satisfies $0 \in \partial f(\x) + N_C(\x)$.
\end{lemma}

\begin{proof}
According to step 7 of  the algorithm 
we have $\widetilde{\rho}_k \geq \widetilde{\gamma}$ for infinitely many $k\in \mathcal K$. Since $R_k$ is never increased during the inner loop, that implies
$R_k\to 0$. Hence $\y^k,\z^k\to \x$ as $k\to \infty$, where we use the trial step generation rule (\ref{trial}) in step 4 of the algorithm. We argue that this implies $\phi_k(\z^k,\x) \to f(\x)$.

Indeed,
$\limsup_{k\to\infty} \phi_k(\z^k,\x) \leq \limsup_{k\to \infty} \phi(\z^k,\x) = \lim_{k\to\infty} \phi(\z^k,\x)=f(\x)$ is always true due to
$\phi_k\leq \phi$ and axiom $(M_1)$, and where $\phi$ is the envelope model of the
cutting plane oracle $\mathscr O$. On the other hand, $\phi_k$ includes (i.e. dominates) an exactness plane $m_0(\cdot,\x)=f(\x) + \g_0^\top(\cdot-\x)$ by rule $(W_1)$, hence  
$f(\x) = \lim_{k\to\infty} m_0(\z^k,\x) \leq \liminf_{k\to\infty} \phi_k(\z^k,\x)$. These two together show $\phi_k(\z^k,\x)\to f(\x)$, and then immediately also 
$\Phi_k(\z^k,\x)\to f(\x)$. We also readily obtain $\phi_k(\y^k,\x) \to f(\x)$ from the link (\ref{trial}) between $\z^k,\y^k$ in step 4 of the algorithm.

We now prove that $\liminf_{k\to \infty}\| \g_k^*\|= 0$,  where $\g_k^*$ are the aggregate subgradients (definition \ref{aggregate}) at the $\y^k$.
Assume on the contrary that $\| \g_k^*\|\geq \eta > 0$ for all $k$. Since $\y^k\to \x$, by (\ref{kolmogoroff}) there exists a constant $\sigma > 0$
such that
for $k$ large enough,
\begin{equation}
\label{denom}
f(\x)-\Phi_k(\z^k,\x) \geq 
\sigma\eta \| \z^k-\x\|.
\end{equation}

Next observe that
since by rule $(W_3)$ all cutting planes at $\z^k$ are integrated in the new model $\phi_{k+1}(\cdot,\x)$, we have $\phi_{k+1}(\z^k,\x) \geq m_k(\z^k,\x)
= a_k + \g_k^\ttop (\z_k-\x)$, where $(a_k,\g_k)\in \mathscr O(\z^k,\x)$. But by  $(\mathscr O_3)$ there exist $\epsilon_k\to 0^+$ and some such plane $m_k(\cdot,\x)$ 
such that $f(\z^k) \leq m_k(\z^k,\x) + \epsilon_k \|\z^k-\x\|$. Hence 
\begin{equation}
\label{num}
f(\z^k) \leq m_{k}(\z^k,\x) + \epsilon_k\|\z^k-\x\| \leq \phi_{k+1}(\z^k,\x) + \epsilon_k \|\z^k - \x\|.
\end{equation}
Now using (\ref{denom}) and (\ref{num})  we estimate
\[
\widetilde{\rho}_k = \rho_k + \frac{f(\z^k)- \phi_{k+1}(\z^k,\x)}{f(\x)-\Phi_k(\z^k,\x)}
\leq \rho_k + \frac{{\epsilon}_k \| \z^k-\x\| }{  \sigma \eta \| \z^k-\x \|} = \rho_k+ {\epsilon}_k/(\sigma\eta).
\]
Since ${\epsilon}_k \to 0$ and $\rho_k < \gamma$, we have $\limsup \widetilde{\rho}_k \leq \gamma < \widetilde{\gamma}$,
a contradiction with $\widetilde{\rho}_k > \widetilde{\gamma}$ for the infinitely many $k\in \mathcal K$. That proves $\g_k^*\to 0$
for a subsequence $k\in \mathcal N\subset \mathcal K$. 

Write $\g_k^*=\p_k+\q_k$
with $\p_k\in \partial_1 \phi_k(\y^k,\x)$ and $\q_k\in N_C(\y^k)$. Using boundedness of the $\y^k$,
and hence boundedness of the $\p_k$, we
extract another subsequence $k\in \mathcal K'\subset \mathcal N$ such that $\p_k\to \p$, $\q_k\to \q$, $\p+\q=0$. Since $\y^k\to \x$,
we have $\q\in N_C(\x)$. We argue that $\p\in \partial f(\x)$. Indeed,
for any test vector $\h$ the subgradient inequality gives
\[
\p_k^\ttop   \h\leq \phi_k(\y^k+\h,\x)-\phi_k(\y^k,\x) \leq
\phi(\y^k+\h,\x)-\phi_k(\y^k,\x).
\]
Since $\phi_k(\y^k,\x)\to f(\x)=\phi(\x,\x)$, passing to the limit using $\y^k\to \x$ gives
\[
\p^\ttop \h\leq \phi(\x+\h,\x)-\phi(\x,\x),
\]
proving $\p\in\partial_1 \phi(\x,\x)\subset \partial f(\x)$ by axiom $(M_1)$. 
This proves $0 = \p + \q \in \partial f(\x) + N_C(\x)$, 
hence $\x$ is a critical point of (\ref{program}).
%
\end{proof}

\begin{remark}
Kiwiel's aggregation rule \cite{kiwiel_aggregation} for the convex bundle method allows to limit
the bundle to any fixed number $\geq 3$ planes. In \cite{NPR:08} we have shown that the rule carries over
to our non-convex bundle method, see also \cite{cutting}.  
Ruszczy\'nski  \cite{rus} had asked whether  aggregation  could also be used in the convex trust region method. 
In \cite{ANR:18} we had justified the rule for the non-convex non-smooth trust region method
under the hypothesis that $Q \succ 0$. Here we shall show that this hypothesis cannot be removed. 
\end{remark}

The following result justifies the use of aggregation in the first place for the special case $\z^k=\y^k$ and $Q \succ 0$.  Note that 
the trivial choice $\z^k=\y^k$ in step 4 is always authorized (due to $\Theta \geq 1$ and $\theta \leq 1$ in  rule (\ref{trial})), but
of course we want to use the additional freedom offered by $\z^k$ to improve performance of our method, so $\z^k=\y^k$
is rather restrictive, and we will remove it later.

\begin{lemma}
\label{long}
Suppose $Q \succ 0$, the inner loop at $\x$ turns infinitely, and the trust-region radius $R_k$ stays bounded away from $0$.
Suppose the $\y^k$ are chosen as trial steps.
Then $\x$ is
a critical point of {\rm (\ref{program})}.
\end{lemma}

\begin{proof}
Since the trust-region radius is frozen $R_k=R_{k_0}$ from some counter $k_0$ onwards, we write $R := R_{k_0}$.
According to step 7 of the algorithm that means $\widetilde{\rho}_k < \widetilde{\gamma}$ for $k\geq k_0$. 
The only progress in the working model
as we update $\phi_k \to \phi_{k+1}$
is now the addition of the cutting plane(s) and the aggregate plane.  The working models $\phi_k$ now contain at least three planes,
an exactness plane, at least one cut from the last unsuccessful trial step, and the aggregate plane. They may  contain more planes, 
but those will not be used in our argument below.

Since $R_k$ stays bounded away
from 0, 
it is not a priori clear whether $\y^k\to \x$, and we have to work to prove this.  
Since $Q \succ 0$, $|\x|_Q^2 = \x^\ttop Q\x$ is a Euclidian norm. We  write the objective
of the tangent program as
\[
\Phi_k(\cdot,\x) = \phi_k(\cdot,\x) + \textstyle\frac{1}{2} | \cdot - \x|_Q^2.
\] 
We know that the  aggregate plane $m_k^*(\cdot,\x)$ satisfies
$m_k^*(\y^k,\x) = \phi_k(\y^k,\x)$, so it memorizes the value $\phi_k(\y^k,\x)$. The latter gives
\begin{equation}
\label{what}
\Phi_k(\y^k,\x) = m_k^*(\y^k,\x) + \textstyle \frac{1}{2} |\y^k-\x|_Q^2.
\end{equation}
Now we introduce the quadratic function
\[
\Phi^*_k(\cdot,\x) = m_k^*(\cdot,\x) + \textstyle \frac{1}{2} |\cdot - \x|_Q^2,
\]
then from what we have just seen in (\ref{what})
\begin{equation}
\label{equal}
\Phi_k^*(\y^k,\x) = \Phi_k(\y^k,\x).
\end{equation}
Moreover, we have
\begin{equation}
\label{monotone}
\Phi_k^*(\cdot,\x) \leq \Phi_{k+1}(\cdot,\x),
\end{equation}
because according to the aggregation rule $(W_2)$ we include the aggregate plane $m_k^*(\cdot,\x)$ in the new model
$\phi_{k+1}$, that is, we have $m_k^*(\cdot,\x) \leq \phi_{k+1}(\cdot,\x)$, and hence (\ref{monotone}).
Expanding the quadratic function $\Phi_k^*(\cdot,\x)$ at $\y^k$ gives
\[
\Phi_k^*(\cdot,\x) = \Phi_k^*(\y^k,\x) + \nabla \Phi_k^*(\y^k,\x)^\ttop (\cdot-\y^k) +\textstyle \frac{1}{2}| \cdot-\y^k|_Q^2,
\]
where $\nabla \Phi_k^*=\g_k^* + Q(\y^k-\x)$. 
From the optimality condition of the tangent program at $\y^k$ we get  $\g_k^*+Q(\y^k-\x) = - \vv_k$ with $\vv_k$ in the normal cone to the 
ball $B(\x,R)$ at $\y^k$, hence 
\begin{equation}
\label{just}
\Phi_k^*(\cdot,\x) = \Phi_k^*(\y^k,\x) - \vv_k^\ttop(\cdot-\y^k) + \textstyle \frac{1}{2} |\cdot - \y^k|_Q^2.
\end{equation}
Now we argue as follows:
\begin{align}
\label{chain}
\begin{split}
\Phi_k(\y^k,\x) &= \Phi_k^*(\y^k,\x)   \hspace*{3.85cm}   \mbox{ (by (\ref{equal})) } \\
&\leq \Phi_k^*(\y^k,\x) + \textstyle\frac{1}{2} |\y^{k+1}-\y^k|_Q^2 \\
&= \Phi_k^*(\y^{k+1},\x) + \vv_k^\ttop (\y^{k+1}-\y^k)  \quad\;\;\, \mbox{ (by (\ref{just}))} \\
&\leq \Phi_{k}^*(\y^{k+1},\x)   \hspace*{3.6cm}   \mbox{(since $\vv_k^\ttop(\y^{k+1}-\y^k) \leq 0$) } \\
&\leq \Phi_{k+1}(\y^{k+1},\x)   \hspace*{3.25cm}   \mbox{(by (\ref{monotone}))} \\
&\leq \Phi_{k+1}(\x,\x)  \qquad\qquad\qquad\qquad\quad\,   \mbox{($\y^{k+1}$ minimizer of $\Phi_{k+1}(\cdot,\x)$)} \\
&= \phi(\x,\x) = f(\x).
\end{split}
\end{align}
Therefore the sequence $\Phi_k(\y^k,\x)$ is increasing and bounded above, and converges to a limit $\Phi^* \leq f(\x)$. Going back with this information to the estimation chain (\ref{chain})
shows $\frac{1}{2}|\y^{k+1}-\y^k|_Q^2 \to 0$ and also $\vv_k^\ttop (\y^{k+1}-\y^k) \to 0$. 
Since  $Q \succ 0$, we deduce
$\y^{k+1}-\y^k\to 0$. 
Then also
\[
\textstyle \frac{1}{2} |\y^{k+1}-\x|_Q^2 - \frac{1}{2} |\y^k-\x|_Q^2 \to 0,
\]
by the triangle inequality.
In consequence
\begin{align}
\label{10}
\begin{split}
&\phi_{k+1}(\y^{k+1},\x) - \phi_k(\y^k,\x) = \Phi_{k+1}(\y^{k+1},\x) - \Phi_k(\y^k,\x) \\
&\hspace{4.5cm} -\textstyle \frac{1}{2}|\y^{k+1}-\x|_Q^2+\frac{1}{2} |\y^k-\x|_Q^2 \to 0.
\end{split}
\end{align}

So far we have not made use of the cutting plane oracle, and we need it now. 
Observe  that one of the cutting planes $m_k(\cdot,\x)$ furnished by the oracle $\mathscr O$  at $\y^k$ must be an affine support function of $\phi_{k+1}(\cdot,\x)$ at $\y^{k}$.
Indeed, $\phi_{k+1}(\cdot,\x)$ is the envelope of the exactness plane, the aggregate plane, and the cutting planes,
but here one of the cutting planes must be active at $\y^k$, i.e.,
$\phi_{k+1}(\y^k,\x) = m_{\y^k,\g_k}(\y^k,\x) = a_k + \g_k^\ttop (\y^k-\x)$ for some $(a_k,\g_k)\in \mathscr O(\y^k,\x)$,
because if the values of the planes furnished
by the oracle are below the values of the other planes maintained in $\phi_{k+1}(\cdot,\x)$, then
$\widetilde{\rho}_{k} \geq 1$. But that 
cannot happen here since we are in the case where $R_k$ is frozen, and we know that  when
$\widetilde{\rho}_k \geq 1$, the trust-region radius $R_k$ is reduced due to the secondary test in step 7.
Hence by the subgradient inequality
\begin{equation}
\label{previous}
\phi_{k+1}(\y^k,\x) + \g_k^\ttop(\cdot-\y^k)\leq \phi_{k+1}(\cdot,\x).
\end{equation}
Now using (\ref{previous}) and the fact that $\phi_k(\y^k,\x) \leq \phi_{k+1}(\y^k,\x)$, we estimate as follows:
\begin{align*}
0 &\leq \phi_{k+1}(\y^k,\x) - \phi_k(\y^k,\x) \\
&= \phi_{k+1}(\y^k,\x) + \g_k^\ttop(\y^{k+1}-\y^k) - \phi_k(\y^k,\x) - \g_k^\ttop(\y^{k+1}-\y^k)\\
&\leq \phi_{k+1}(\y^{k+1},\x) - \phi_k(\y^k,\x) - \g_k^\ttop(\y^{k+1}-\y^k).
\end{align*}
Since the $\g_k$ are bounded and $\y^{k+1}-\y^k\to 0$, we have $\g_k^\ttop(\y^{k+1}-\y^k)\to 0$, hence
using (\ref{10}) we deduce $\phi_{k+1}(\y^{k},\x)-\phi_k(\y^k,\x) \to 0$, and also $\Phi_{k+1}(\y^k,\x) - \Phi_k(\y^k,\x) \to 0$.

Now we claim that  $\phi_k(\y^k,\x)\to f(\x)$. Since
$\phi_k(\y^k,\x)\leq \Phi_k(\y^k,\x) \to \Phi^* \leq f(\x)$, it remains to prove $\liminf \phi_k(\y^k,\x) \geq f(\x)$.
Suppose that this is not the case, and let $\phi_k(\y^k,\x) \to f(\x) - \eta$ for a subsequence and some $\eta > 0$.
Then also $\phi_{k+1}(\y^k,\x) \to f(\x)-\eta$ for that subsequence  (using
$\phi_k(\y^k,\x)-\phi_{k+1}(\y^k,\x)\to 0$ proved above).  Passing to yet another subsequence,  and using boundedness of the $\y^k$, we may assume
$\frac{1}{2}|\y^k-\x|_Q^2 \to \ell \geq 0$.
Choose $\delta > 0$ such that $\delta < (1-\widetilde{\gamma})\eta$. From what we have just seen there exists $k_1$ such that
\[
\phi_{k+1}(\y^k,\x) - \phi_k(\y^k,\x) < \delta
\]
for all $k \geq k_1$. Now recall that $\widetilde{\rho}_k \leq \widetilde{\gamma}$ for every $k\geq k_0$, hence
\[
\widetilde{\gamma} \left(  \Phi_k(\y^k,\x) - f(\x) \right)
\leq \phi_{k+1}(\y^k,\x) - f(\x)
\leq \phi_k(\y^k,\x)-f(\x) + \delta 
. 
\]
Passing to the limit 
gives $-\widetilde{\gamma} \eta + \widetilde{\gamma} \ell \leq -\eta +  \delta$, hence $(1-\widetilde{\gamma})\eta + \ell \widetilde{\gamma} \leq \delta$,
which contradicts the choice of $\delta$.
Hence $\phi_k(\y^k,\x) \to f(\x)$.  We immediately deduce that $\Phi_k(\y^k,\x)\to f(\x)$ and $\Phi_{k+1}(\y^k,\x)\to f(\x)$.
This also shows
$|\y^k-\x|_Q \to 0$, as follows from the estimates
\[
\phi_k(\y^k,\x) \leq
\Phi_k(\y^k,\x) =\phi_k(\y^k,\x) + \textstyle\frac{1}{2}|\y^k-\x|_Q^2 \leq \Phi^* \leq f(\x),
\]
where due to $\phi_k(\y^k,\x)\to f(\x)$ all terms go to $f(\x)$, implying $|\y^k-\x|_Q \to 0$.

Let us now show that $\x$ must be a critical point of (\ref{program}).
Let $\Phi(\cdot,\x) = \phi(\cdot,\x) +\frac{1}{2} |\cdot - \x|_Q^2$ be the second order model associated with the envelope model $\phi$ of $\mathscr O$.
Let $\y \in C \cap B(\x,R)$, then $\phi(\y,\x) \geq \phi_k(\y,\x)$ by construction of the working models, hence
$\Phi(\y,\x) \geq \Phi_k(\y,\x) \geq \Phi_k(\y^k,\x)$, the latter since $\y^k$ is the minimizer of $\Phi_k(\cdot,\x)$ over $C \cap B(\x,R)$. 
Since $\Phi_k(\y^k,\x) \to f(\x)$ by what we have seen above, we deduce
$\Phi(\y,\x) \geq f(\x)$ for every $\y\in C \cap B(\x,R)$. But $\Phi(\x,\x) = f(\x)$, which means $\x$ is a minimizer
of $\Phi(\cdot,\x)$ over $C \cap B(\x,R)$, hence over $C$.  In consequence there exists $\g\in \partial_1\Phi(\x,\x)$ and 
$\vv \in N_C(\x)$ such that $\g+\vv=0$. 
Since $\partial_1\Phi(\x,\x)=\partial_1 \phi(\x,\x) \subset \partial f(\x)$, this proves that
$\x$ is a critical point of (\ref{program}). 
\end{proof}

\begin{remark}
\label{fall_back}
Let us now see our first method to justify
aggregation in the case $Q \succ 0$,
$\z^k\not=\y^k$.
In the algorithm we first
allow $\z^k$ as a trial point in  step 4. If acceptance 
in step 5 fails for $z^k$,
then we include the cutting plane at $\z^k$ and continue with step 7 for $\z^k$. However, if step 7 for $\z^k$ gives no reduction of $R_k$, then we are in the difficult case
not covered by Lemma \ref{short}.
We then do the following. We fall back on $\y^k$
as the trial point, i.e., we forget about $\z^k$. 
When $\y^k$ is not accepted, we proceed with step 6, now for $\y^k$, and apply aggregation.  This is  justified,
because we are  in the situation covered by Lemma \ref{long}. Note that the additional work required in steps 6 and 7 is marginal, so we do not 
waste time by this  evasive maneuver.  We could even perform this maneuver as default (i.e. checking $\y^k$ whenever $\z^k$ fails). 
We have proved the following
\end{remark}

\begin{lemma}
\label{black_lemma}
Suppose the inner loop at $\x$ turns infinitely.
Let $Q \succ 0$ and accept to fall back on $\y^k$ if $\z^k$ fails in step {\rm 5} with $\widetilde{\rho}_k < \widetilde{\gamma}$ in step {\rm 7}.
 Then $0\in \partial f(\x)+N_C(\x)$.
 \hfill $\square$
\end{lemma}

Let us now discuss a second way to justify aggregation in the case $Q \succ 0$, $z^k \not= y^k$, which does {\em not} require
the falling back procedure above. 
The idea proposed in \cite{hertlein} is to sharpen condition (\ref{trial}) as follows. Fix a
sequence $\nu_k \to 0^+$ and require that the trial points $z^k$ satisfy
\begin{equation}
\label{strong_trial}
z^k \in B(y^k,\nu_k) \mbox{ and } (\ref{trial}).
\end{equation} 
Then the technique of Lemma \ref{long} allows a direct justification
of aggregation with the aggregate plane taken at $y^k$, and the cutting plane taken at $z^k$.

\begin{lemma}
\label{blue_lemma}
Suppose trial points $z^k$ in step {\rm 4} of the algorithm satisfy the stronger condition {\rm (\ref{strong_trial})}. Suppose in case of a null step $z^k$
cutting planes are generated at $z^k$ and the aggregate plane is generated at $y^k$. Let $Q \succ 0$. Suppose the inner loop turns infinitely. Then
$x$ is a critical point of {\rm (\ref{program})}.
\end{lemma}

\begin{proof}
We follow the line of proof in Lemma \ref{long}. Up to formula (\ref{10}) only properties of the aggregate plane are used, so we arrive at that same formula.
Now we use that $\|z^k-y^k\|\leq \nu_k\to 0$, and then since the $\phi_k(\cdot,x)$ have a joint Lipschitz constant $L$ on the bounded set $B(x,R)$,
we deduce $|\phi_{k}(z^{k},x)  - \phi_k(y^k,x)| \leq L \|z^k-y^k\| \to 0$. In view of (\ref{10}) this implies
$$
\phi_{k+1}(z^{k+1},x) - \phi_k(z^k,x) \to 0.
$$
Now observe that we must have $\phi_{k+1}(z^k,x) \geq \phi_k(z^k,x)$, because otherwise $\widetilde{\rho}_k < \widetilde{\gamma}$ could
not be satisfied. Hence one of the cutting planes $m_k(\cdot,x)$ is an affine support plane of $\phi_{k+1}(\cdot,x)$ at $z^k$, i.e.,
$\phi_{k+1}(z^k,x) = m_k(z^k,x) = a_k + g_k^\ttop(z^k-x)$ for at least one $(a_k,g_k)\in \mathscr O(z^k,x)$.  Then from the subgradient inequality at $z^k$, 
\[
\phi_{k+1}(z^k,x) + g_k^\ttop (\cdot - x) \leq \phi_{k+1}(\cdot,x).
\] 
Consequently, if we apply this at $z^{k+1}$, 
\begin{align*}
0 & \leq \phi_{k+1}(z^k,x) - \phi_k(z^k,x) \\
&= \phi_{k+1}(z^k,x) + g_k^\ttop(z^{k+1}-z^k) - \phi_k(z^k,x) - g_k^\ttop(z^{k+1}-z^k) \\
&\leq \phi_{k+1}(z^{k+1},x) - \phi_k(z^k,x) - g_k^\ttop (z^{k+1}-z^k).
\end{align*}
Since $y^{k+1}-y^k\to 0$ by (\ref{chain}) and $z^k-y^k\to 0$ by (\ref{strong_trial}), we deduce $z^{k+1}-z^k\to 0$. Since the $g_k$ are bounded, we have $g_k^\ttop(z^{k+1}-z^k)\to 0$. On the other hand, we know that
$\phi_{k+1}(z^{k+1},x)-\phi_k(z^k,x)\to 0$, so altogether we deduce 
\begin{equation}
\label{new2}
\phi_{k+1}(z^k,x) - \phi_k(z^k,x)\to 0,
\end{equation}
and then also $\Phi_{k+1}(z^k,x)-\Phi_k(z^k,x)\to 0$. 

Now we prove that $\phi_k(z^k,x)\to f(x)$. 
Since $\phi_k(z^k,x) - \phi_k(y^k,x)\to 0$ and by (\ref{chain}) $\phi_k(y^k,x) \to \Phi^* \leq f(x)$, we have
$\lim_k \phi_k(z^k,x)=\Phi^* \leq f(x)$, and it remains to show the opposite estimate.

Suppose that contrary to what is claimed $\phi_k(z^k,x) \to f(x) - \eta$ for some $\eta > 0$ and a subsequence. Then also $\phi_{k+1}(z^k,x)\to f(x)-\eta$
as a consequence of (\ref{new2}).  Passing to yet another subsequence,  and using boundedness of the $\z^k$, we may assume
$\frac{1}{2}|\z^k-\x|_Q^2 \to \ell \geq 0$.
Choose $\delta > 0$ such that $\delta < (1-\widetilde{\gamma})\eta$. From what we have just seen there exists $k_1$ such that
\[
\phi_{k+1}(\z^k,\x) - \phi_k(\z^k,\x) < \delta
\]
for all $k \geq k_1$. Now recall that $\widetilde{\rho}_k \leq \widetilde{\gamma}$ for every $k\geq k_0$, hence
\[
\widetilde{\gamma} \left(  \Phi_k(\z^k,\x) - f(\x) \right)
\leq \phi_{k+1}(\z^k,\x) - f(\x)
\leq \phi_k(\z^k,\x)-f(\x) + \delta 
. 
\]
Passing to the limit 
gives $-\widetilde{\gamma} \eta + \widetilde{\gamma} \ell \leq -\eta +  \delta$, hence $(1-\widetilde{\gamma})\eta + \ell \widetilde{\gamma} \leq \delta$,
which contradicts the choice of $\delta$.
Hence $\phi_k(\z^k,\x) \to f(\x)$.  

Using $z^k-y^k\to 0$, we deduce
$\phi_k(y^k,x) \to f(x)$ and then  from (\ref{chain}) we get $\Phi_k(y^k,x) \to f(x)$. That implies $|y^k-x|_Q^2 \to 0$, and from $Q\succ 0$ follows
$y^k\to x$. 

Let us now show that $\x$ must be a critical point of (\ref{program}).
Let $\Phi(\cdot,\x) = \phi(\cdot,\x) +\frac{1}{2} |\cdot - \x|_Q^2$ be the second order model associated with the envelope model $\phi$ of $\mathscr O$.
Let $\y \in C \cap B(\x,R)$, then $\phi(\y,\x) \geq \phi_k(\y,\x)$ by construction of the working models, hence
$\Phi(\y,\x) \geq \Phi_k(\y,\x) \geq \Phi_k(\y^k,\x)$, the latter since $\y^k$ is the minimizer of $\Phi_k(\cdot,\x)$ over $C \cap B(\x,R)$. 
Since $\Phi_k(\y^k,\x) \to f(\x)$ by what we have seen above, we deduce
$\Phi(\y,\x) \geq f(\x)$ for every $\y\in C \cap B(\x,R)$. But $\Phi(\x,\x) = f(\x)$, which means $\x$ is a minimizer
of $\Phi(\cdot,\x)$ over $C \cap B(\x,R)$, hence over $C$.  In consequence there exists $\g\in \partial_1\Phi(\x,\x)$ and 
$\vv \in N_C(\x)$ such that $\g+\vv=0$. 
Since $\partial_1\Phi(\x,\x)=\partial_1 \phi(\x,\x) \subset \partial f(\x)$, this proves that
$\x$ is a critical point of (\ref{program}). 
\end{proof}

\begin{example}
We show by way of an example that the hypothesis $Q \succ 0$ cannot be removed neither from Lemmas \ref{long},\ref{black_lemma}, nor from Lemma \ref{blue_lemma}.
Consider (\ref{program}) with  $f(\x) = \frac{1}{2}x_1^2-x_1 + \frac{1}{4} x_2^2$, $C=\mathbb R^2$. 
In the algorithm choose $\gamma=\frac{1}{2}$,
$Q=0$,  the trust region norm $\|\cdot\| = |\cdot|_\infty$. Suppose the current iterate is $\x=(x_1,x_2)=(0,0)$ with $f(0,0)=0$, $\g_0=(-1,0)\in \partial f(0,0)$. The exactness plane
is $m_0(\z,\x) = -z_1$, and we let $\phi_1=m_0$. Suppose the current trust region radius is $R_1=1$,
then a minimizer of the tangent program is $\y^1 = (1,1)$. In fact, the entire segment $\{1\} \times [-1,1]$ is minimizing, but we choose $\y^1$ at a corner. Take $z^1=y^1$.
The predicted progress at $\y^1$ is $1$, because $m_0(\y^1,\x)=\phi_1(\y^1,\x)=-1$.
The aggregate plane $m_1^*(\cdot,\x)$ at $\y^1$
coincides with $m_0(\cdot,\x)$, because there is only one plane in the working model which can be active.

\begin{pspicture}(-4,-5)(4,2)

\psscalebox{1.5}{
\psplot[linecolor=blue,plotstyle=dots,
      plotpoints=250,dotstyle=*,dotsize=1pt]
      {-2}{2}{x 2 exp 0.25 mul  0.5 sub} \uput[0](1.3,.8){\tiny ${\color{blue}f(1,\cdot)}$}

\psline[arrows=->,arrowsize=4pt 1](-2,0)(2.5,0)
\uput[0](2.4,0){\tiny $y_2\qquad $    cut   $[y_1=1]$}
\psline(-1,-1)(1,-1) \uput[0](1,-1){\tiny $m_0$}
\psline(-1,-.1)(-1,.1) \uput[0](-1.45,.3){-\tiny$1$}
\psline(0,-.1)(0,.1)
\psline(1,-.1)(1,.1)  \uput[0](0.8,.3){\tiny$1$}

\psline(-.5,-1)(1,-.25)(2.5,.5)
\psline[linestyle=dotted](-.5,-1)(-1,-1.25) 
\uput[0](-1.3,-1.4){\tiny $m_1$}
\psdot(-1,-1)
\psdot(-.5,-1)
\psdot(1,-.25)
\psline(1,-1.1)(1,-0.9)
}
\end{pspicture}

We have $f(\y^1)=-\frac{1}{4}$. Therefore $\rho_1 = \frac{0-(-\frac{1}{4})}{0-(-1)} =\frac{1}{4} < \gamma$
so $\y^1$ is a null step. The function being convex, the cutting plane at $\y^1$ is the tangent plane
$m_1(\z,\x)= \frac{1}{4} +\frac{1}{2}(z_2-1)$.
Since $f$ is convex, the cutting plane always gives an improvement due to $\rho=\widetilde{\rho}$, so we keep
$R_2=R_1$, and we have
$\phi_2(\cdot,\x) = \max\{m_0(\cdot,\x),m_1(\cdot,\x)\}=\max\left\{-z_1,\frac{1}{4}+\frac{1}{2}(z_2-1)\right \}$. 
Now a minimizer of $\phi_2(\cdot,\x)$ over the trust region is $\y^2=(1,-1)$. 
(The figure shows the trace of the graph of $\phi_1,\phi_2,f$ in the plane $y_1=1$).
Since
$m_0((1,-1),\x)=-1$ and $m_1((1,-1),\x) = -\frac{5}{4}$,  only $m_0$ is active at $\y^1$, hence the aggregate plane is again $m_2^*=m_0$.
Since $\rho_2=\widetilde{\rho}_2 = \frac{1}{4}$, the situation is the same as in the first trial, (the figure being the one reflected by $y_2=0$),
and we need a cutting plane at $\y^2$. 
This is now $m_2(\z,\x) = -\frac{1}{4} -\frac{1}{2}(z_2+1)$. Now if we do not keep the plane $m_1$ in the working model, the tangent program is minimizing
$\phi_3(\cdot,\x) = \max\{m_0(\cdot,\x),m_2(\cdot,\x)\}=\max\{-z_1,-\frac{1}{4}-\frac{1}{2}(z_2+1)\}$ over $[-1,1]^2$,
and a solution is $\y^3=(1,1)$. This means, the inner loop oscillates between $\y^{2k+1}=\y^1$ and $\y^{2k}=\y^2$,
and no progress occurs.

\end{example}


\begin{remark}
The phenomenon in this example cannot occur if $\|\cdot\| =|\cdot|_2$. However, since the choice of this norm is not
as useful as in the classical trust region method due to the affine constraints in $\phi_k$, we do not present the details
of the result here.
\end{remark}

\begin{remark}
In the general case $Q\succeq 0$ we can still limit the size of the working model to
$n+2$ using Carathéodory's theorem. See \cite{ANR:2016} for the details.
\end{remark}

\begin{remark}
In summary, justification of the aggregate rule in the case  $Q\succ 0$ can be based either on the technique
of remark \ref{fall_back}, or on adopting the stronger rule (\ref{strong_trial}) for trial steps.         
\end{remark}

\begin{remark}
We investigate whether
backtracking steps $\z_\alpha = \x + \alpha(\y^k-\x)$ for $0 < \alpha < 1$ could be trial steps in the sense
of (\ref{trial}) or (\ref{strong_trial}), as this would allow to use linesearch in case of rejection of $y^k$.

Let $\Delta := f(x) -\Phi_k(y^k,x)>0$. 
By convexity the line joining $(x,f(x))$ and $(y^k,f(x)-\Delta)$ is above the curve 
$\alpha \mapsto \Phi_k(z_\alpha,x)$.    Therefore $z_\theta$ satisfies $f(x) - \Phi_k(z_\theta,x) \geq \theta \Delta$. Since
$z_\theta \in B(x,\|x-y^k\|)$ is clear, every $z_\alpha$ with
$\theta \leq \alpha \leq 1$ is a trial point in the sense of (\ref{trial}). But not necessarily in the sense of (\ref{strong_trial}),
because $\|z_\theta - y^k\| = (1-\theta) \|x-y^k\|$.  If we impose $\|z_\theta-y^k\| \leq \nu_k$, then shorter steps might be forced.
This is an argument for the fall back method.
\end{remark}

\subsection{Convergence of the outer loop}
We are now ready to prove the main convergence result for our optimization method. 
This is where strictness of the oracle $(\widehat{\mathscr O}_3)$ is needed.

\begin{theorem}
Suppose $\x^1\in C$ is such that the level set $\{\x\in C: f(\x) \leq f(\x^1)\}$ is bounded.  Let $\x^j\in C$ be the sequence of serious
iterates generated by the bundle trust-region algorithm based on a strict cutting plane oracle. Then every accumulation point $\x^*$ of the $\x^j$ is a critical point of 
{\rm (\ref{program})}.
\end{theorem}

\begin{proof}
1)
Without loss we consider the case where the algorithm generates an infinite sequence $\x^j\in C$
of serious iterates. Suppose that at  outer loop counter $j$ the inner loop finds a successful trial
step at inner loop counter $k_j$, that is, $\z^{k_j} = \x^{j+1}$, where the corresponding
solution of the tangent program is $\tilde{\x}^{j+1}=\y^{k_j}$. Then $\rho_{k_j}\geq \gamma$, which means
\begin{equation}
\label{descent1}
f(\x^j) - f(\x^{j+1}) \geq \gamma \left( f(\x^j) - \Phi_{k_j}(\x^{j+1},\x^j) \right).
\end{equation}
Moreover, by condition (\ref{trial}) we have $\|\tilde{\x}^{j+1}-\x^j\|\leq \Theta \|\x^{j+1}-\x^j\|$ and
\begin{equation}
\label{descent1a}
f(\x^j)-\Phi_{k_j}(\x^{j+1},\x^j) \geq \theta \left( f(\x^j)-\Phi_{k_j}(\tilde{\x}^{j+1},\x^j)  \right),
\end{equation}
and combining (\ref{descent1}) and (\ref{descent1a}) gives
\begin{eqnarray}
\label{descent}
f(\x^j)-f(\x^{j+1}) \geq \gamma\theta \left( f(\x^j)-\Phi_{k_j}(\tilde{\x}^{j+1},\x^j)  \right).
\end{eqnarray}
Since $\y^{k_j}=\tilde{\x}^{j+1}$ is a solution of the $k_j^{\rm th}$ tangent program (\ref{tangent}) of the
$j^{\rm th}$ inner loop, by the optimality condition there exist
$\g_j^*\in \partial \left( \phi_{k_j}(\cdot,\x^j)+i_C \right)(\tilde{\x}^{j+1})$ and a unit normal vector
$\vv_j$ to the ball $B(\x^j,R_{k_j})$ at $\tilde{\x}^{j+1}$ such that 
\begin{equation}
\label{kt}
\g_j^* + Q_j(\tilde{\x}^{j+1}-\x^j) + \lambda_j \vv_j = 0,
\end{equation}
where $\lambda_j = \|\g_j^*+Q_j (\tilde{\x}^{j+1}-\x^j)\|$.
Decomposing further, there exist
$\p_j\in \partial_1 \phi_{k_j}(\tilde{\x}^{j+1},\x^j)$ and $\q_j\in N_C(\tilde{\x}^{j+1})$
such that
\begin{equation}
\label{opt}
0 = \g_j^* + Q_j(\tilde{\x}^{j+1}-\x^j)+\lambda_j\vv_j=\p_j +  \q_j +  Q_j(\tilde{\x}^{j+1}-\x^j) + \lambda_j \vv_j.
\end{equation}
By the subgradient inequality,
applied to $\p_j\in \partial  \phi_{k_j}(\cdot,\x^j)(\tilde{\x}^{j+1})$, we have
\begin{align*}
-(\q_j+\lambda_j &v_j)^\ttop  (\x^j   -\tilde{\x}^{j+1})+ |\tilde{\x}^{j+1}-\x^j|_{Q_j}^2 \\ 
& = \p_j^{\ttop} (\x^j-\tilde{\x}^{j+1}) \qquad\qquad \qquad\qquad \; \mbox{ (using (\ref{opt}))}\\
&\leq
\phi_{k_j}(\x^j,\x^j) - \phi_{k_j}(\tilde{\x}^{j+1},\x^j)   \qquad \quad\mbox{ (subgradient inequality)}\\
&= f(\x^j) - \phi_{k_j}(\tilde{\x}^{j+1},\x^j)    \\
&=f(\x^j) - \Phi_{k_j}(\tilde{\x}^{j+1},\x^j) +\textstyle \frac{1}{2}|\tilde{\x}^{j+1}-\x^j|_{Q_j}^2 \\
&\leq \gamma^{-1} \theta^{-1}\left( f(\x^j)-f(\x^{j+1})\right) + \textstyle\frac{1}{2}|\tilde{\x}^{j+1}-\x^j|_{Q_j}^2 \quad\; \mbox{(by (\ref{descent}))}.
\end{align*}
Re-arranging,  we obtain
\begin{equation}
\label{ncc}
(\q_j+\lambda_jv_j)^\ttop  (\tilde{\x}^{j+1}-\x^j)+ \textstyle\frac{1}{2} |\tilde{\x}^{j+1}-\x^j|_{Q_j}^2 
\leq \gamma^{-1} \theta^{-1}\left( f(\x^j)-f(\x^{j+1})\right) .
\end{equation}
By the normal cone condition for $C\cap B(x^j,R_{k_j})$ at $\tilde{\x}^{j+1}$ we have $(\q_j+\lambda_jv_j)^\ttop  (\tilde{\x}^{j+1}-\x^j) \geq 0$, 
hence both terms 
on the left of (\ref{ncc}) are non-negative.
But the term on the right of (\ref{ncc}) is telescoping, hence summable due to convergence of $f(\x^j)$, so we deduce 
summability of $\sum_{j\in \mathbb N} (\q_j+\lambda_jv_j)^{T} (\tilde{\x}^{j+1}-\x^j)<\infty$ and of $\sum_{j\in \mathbb N} |\tilde{\x}^{j+1}-\x^j|_{Q_j} ^2 < \infty$.
From (\ref{opt}) we then get summability of 
$\sum_{j\in \mathbb N}\p_j^\ttop (\x^j-\tilde{\x}^{j+1})$.
Hence
$\p_j^{\ttop} (\x^j-\tilde{\x}^{j+1})\to 0$,  $(\q_j+\lambda_jv_j)^\ttop(\x^j-\tilde{\x}^{j+1})\to 0$,   
$|\x^j-\tilde{\x}^{j+1}|_{Q_j} \to 0$.
Moreover, we know by local boundedness of the subdifferential that the sequence $\p_j$ is bounded, hence the sequence 
$\q_j + \lambda_j \vv_j$ is also bounded.
\\

2)
Now let $\x^*$ be an accumulation point of the sequence $\x^j$. We have to show that $\x^*$ is a critical point of (\ref{program}). Fix a subsequence $j\in J$ such that
$\x^j \to \x^*$, and $\p_j \to \p$, $\q_j+\lambda_j \vv_j \to \q$, $\tilde{\x}^{j+1} \to \tilde{\x}$ for suitable limits $\p,\q,\tilde{\x}$.
We shall now analyze two types of infinite subsequences $j\in J$, those where the trust-region constraint is active at  $\tilde{\x}^{j+1}$
and the Lagrange multiplier of the trust-region constraint is nonzero, i.e. $\lambda_j >0$ in (\ref{kt}),
and those where the Lagrange multiplier of the trust-region constraint vanishes, i.e.,  $\lambda_j=0$ in (\ref{kt}). 
\\

3)
Let us start with the simpler  case of an infinite subsequence $\x^j$, $j\in J$, where 
the Lagrange multiplier of the trust-region constraint vanishes, i.e., $\lambda_j=0$ in (\ref{kt}). That occurs either when
$\|\x^j-\tilde{\x}^{j+1}\| < R_{k_j}$,
i.e., where the trust-region constraint is inactive at acceptance, or when it is active but with vanishing multiplier. 
In this case equation (\ref{opt}) simplifies to
\[
0 = \g_j^* + Q_j(\tilde{\x}^{j+1}-\x^j)=\p_j +  \q_j +  Q_j(\tilde{\x}^{j+1}-\x^j).\]
This means $\p_j \to \p$, $\q_j \to \q$ and $\p + \q = 0$, bearing in mind that we have $|\tilde{\x}^{j+1}-\x^j|_{Q_j}\to 0$.

Now
let ${\h}$ be any test vector, then from the subgradient inequality, the acceptance condition, and using $\phi_k \leq \phi$, we have
\begin{align*}
\p_j^{\ttop} {\h} &\leq \phi_{k_j}(\tilde{\x}^{j+1}+{\h},x^j) - \phi_{k_j}(\tilde{\x}^{j+1},\x^j) \\
&\leq \phi(\tilde{\x}^{j+1}+{\h},\x^j) - f(\x^j) + f(\x^j) - \phi_{k_j}(\tilde{\x}^{j+1},\x^j) \\
&= \phi(\tilde{\x}^{j+1}+{\h},\x^j) - f(\x^j) + f(\x^j) - \Phi_{k_j}(\tilde{\x}^{j+1},\x^j) +\textstyle \frac{1}{2}|\x^j-\tilde{\x}^{j+1}|_{Q_j}^2\\
&\leq \phi(\tilde{\x}^{j+1}+{\h},\x^j) - f(\x^j) + \gamma^{-1}\theta^{-1} \left( f(\x^j)-f(\x^{j+1}) \right) +\textstyle \frac{1}{2}|\x^j-\tilde{\x}^{j+1}|_{Q_j}^2.
\end{align*}
Let ${\h}'$ be another test vector and put ${\h} = \x^j-\tilde{\x}^{j+1}+{\h}'$. 
On substituting this expression above we obtain
\begin{align*}
&\p_j^{\ttop} (\x^j-\tilde{\x}^{j+1}) + \p_j^{\ttop} {\h}' \leq \phi(\x^j+{\h}',\x^j) - f(\x^j)\\
& \hspace{3cm} + \gamma^{-1}\theta^{-1} \left(  f(\x^j)-f(\x^{j+1})\right)+\textstyle \frac{1}{2}|\x^j-\tilde{\x}^{j+1}|_{Q_j}^2.
\end{align*}
Passing to the limit $j\in J$, 
we have $\p_j^{\ttop} (\x^j-\tilde{\x}^{j+1})\to 0$ and $|\x^j-\tilde{\x}^{j+1}|_{Q_j}\to 0$ by 1) above, and
$f(\x^j)-f(\x^{j+1})\to 0$ by the construction of the descent method. 
Moreover, $\limsup_{j\in J} \phi(\x^j+{\h}',\x^j) \leq \phi(\x^*+{\h}',\x^*)$
by $\x^j\to \x^*$ and axiom $(M_3)$.  Since $\p_j\to \p$, we get
\[
\p^{\ttop} {\h}' \leq \phi(\x^*+{\h}',\x^*) - f(\x^*)=\phi(\x^*+{\h}',\x^*)-\phi(\x^*,\x^*).
\]
Since ${\h}'$ was arbitrary and $\phi(\cdot,\x^*)$ is convex, we deduce
$\p\in \partial_1 \phi(\x^*,\x^*)$, hence $\p\in \partial f(\x^*)$ by axiom $(M_1)$.

Now  we have to show that $\q\in N_C(\x^*)$. Since $\q_j^\ttop (\x^j-\tilde{\x}^{j+1})\to 0$ due to $\lambda_j=0$ and by part 1) above, 
and since $\q_j \to \q$, 
we have $\q^\ttop (\x^*-\tilde{\x})=0$. 
Since $\q_j \in N_C(\tilde{\x}^{j+1})$ and $\tilde{\x}^{j+1} \to \tilde{\x}$, we know that
$\q \in N_C(\tilde{\x})$. 
Hence for any element
$\x\in C$ we have $\q^\ttop(\tilde{\x}-\x) \geq 0$ by the normal cone criterion. Hence
$\q^\ttop (\x^*-\x) = \q^\ttop(\tilde{\x}-\x) + \q^\ttop (\x^*-\tilde{\x}) =  \q^\ttop(\tilde{\x}-\x) \geq 0$,
so the normal cone criterion holds also at $\x^*$, proving $\q\in N_C(\x^*)$. We have shown that
$0 =\p + \q \in \partial \left( \phi(\cdot,\x^*)+i_C \right)(\x^*) \subset \partial f(\x^*) + N_C(\x^*)$, hence
$\x^*$ is a critical point of (\ref{program}).
\\

4) Let us now consider the more complicated case of an infinite subsequence,
where $\|\x^j-\tilde{\x}^{j+1}\|=R_{k_j}$ with
$\lambda_j >0$, corresponding to the case of a non-vanishing multiplier
in (\ref{kt}). 
Recall that $\x^j\to \x^*$, $j\in J$, and that
we have to show that $\x^*$ is critical.

We shall now have to distinguish two subcases.
Either $R_{k_j}\geq R_0 >0$ for some $R_0>0$ and all $j\in J$,
or there exists a subsequence $J'\subset J$ such that $R_{k_j}\to 0$ as $j\in J'$.
The first case is discussed in 5), the second case will be handled in 6) - 7).

5)
Let us 
consider the sub-case of an infinite subsequence $j\in J$ where $\|\x^j-\tilde{\x}^{j+1}\|=R_{k_j} \geq R_0 > 0$ for every $j\in J$. 
Recall from the general considerations in part 1) that $|\tilde{\x}^{j+1}-\x^j|_{Q_j} \to 0$. Since on the other hand we are in the case where
$\|\tilde{\x}^{j+1}-\x^j\| = R_{k_j} \geq R_0 > 0$ for $j\in J$, Lemma \ref{est} provides a constant $\sigma > 0$
independent of $j$ such that
\[
f(\x^j) - \Phi_{k_j}(\tilde{\x}^{j+1},\x^j) \geq \sigma \|\g_j^*\| \|\tilde{\x}^{j+1}-\x^j\| \geq \sigma R_0 \|\g_j^*\|,
\]
for all $j \in J$. By acceptance $\rho_{k_j} \geq \gamma$  and (\ref{trial}) that gives 
$$f(\x^j)-f(\x^{j+1}) \geq \gamma\theta \left(f(\x^j) - \Phi_{k_j}(\tilde{\x}^{j+1},\x^j)\right)  \geq \gamma\theta\sigma R_0 \|\g_j^*\|$$
and implies  $\g_j^*\to 0$. Splitting $\g_j^* = \p_j + \q_j \to 0$ with $\p_j\in \partial_1 \phi_{k_j}(\tilde{\x}^{j+1},\x^j)$
and $\q_j \in N_C(\tilde{\x}^{j+1})$,  recall that $\p_j \to \p$ and $\q_j + \lambda_j \vv_j \to \q$, and 
$0 = \g_j^* + Q_j(\tilde{\x}^{j+1}-\x^j) + \lambda_j\vv_j \to \p + \q$, hence $\p + \q = 0$ and $\lambda_j \vv_j \to 0$, i.e., $\lambda_j \to 0$, 
and we find $\q_j \to \q$. Since $\tilde{\x}^{j+1} \to \tilde{\x}$, we infer $\q \in N_C(\tilde{\x})$.
From here on we can argue as in part 2), which means $\x^*$ is critical, where in the last argument we infer $\q_j^\ttop(\x^j-\tilde{\x}^{j+1})\to 0$
from $\lambda_j \to 0$.
\\

6)
It remains to discuss the most complicated sub-case of an infinite subsequence $j\in J$, where the
trust-region constraint is active with non-vanishing multiplier $\lambda_j > 0$
and $R_{k_j}\to 0$. This needs two sub-sub-cases.
The first of these is a sequence $j\in J$ where in each $j^{\rm th}$ outer loop the trust-region radius 
was reduced at least once. The second sub-sub-case are infinite subsequences 
where  the trust-region radius stayed frozen ($R_j^\sharp = R_{k_j}$) throughout the $j^{\rm th}$ inner loop
for every $j\in J$. This is discussed in 6) below.

Let us first consider the case of an infinite sequence $j\in J$  where $R_{k_j}$ is active
at $\tilde{\x}^{j+1}$ with $\lambda_j > 0$,  and  $R_{k_j}\to 0$, $j\in J$, and  where
during the $j^{\rm th}$ inner loop the trust-region radius was reduced at least once.
Suppose this happened  the last time before acceptance at inner loop counter $k_j-\nu_j$
for some $\nu_j\geq 1$.
Then for $j\in J$,
\[
R_{k_j} = R_{k_j-1} = \dots = R_{k_j-\nu_j+1} =\textstyle \frac{1}{2}R_{k_j-\nu_j}.
\]
By step 7 of the algorithm, that implies
\[
\widetilde{\rho}_{k_j-\nu_j} \geq \widetilde{\gamma}, \quad \rho_{k_j-\nu_j} < \gamma.
\]
Now $\|\tilde{\x}^{j+1}-\x^j\| = R_{k_j}$, $\|\x^{j+1}-\x^j\| \leq \Theta R_{k_j}$ and $\|\z^{k_j-\nu_j}-\x^j\|\leq \Theta R_{k_j-\nu_j-1} = 2\Theta R_{k_j}$, hence
$\x^{j+1}-\z^{k_j-\nu_j} \to 0$, $\x^j-\z^{k_j-\nu_j}\to 0$, $j\in J$.

By strictness $(\widehat{\mathscr O}_3)$, and due to convergence $\z^{k_j-\nu_j}\to \x^*$ as well as $\x^j\to \x^*$, 
there exists at least one cutting plane
$m_j(\cdot,\x^j) = a_j + \g_j^\ttop (\cdot-\x^j)$ with $(a_j,\g_j)\in \mathscr O(\z^{k_j-\nu_j},\x^j)$ in tandem with
$\epsilon_j\to 0^+$ such that  $f(\z^{k_j-\nu_j}) \leq m_j(\z^{k_j-\nu_j},\x^j) + \epsilon_j \|\z^{k_j-\nu_j}-\x^j\|$.
Since by rule $(W_3)$ this cutting plane is included in the next working model $\phi_{k_j-\nu_j+1}(\cdot,\x^j)$, we have $m_j(\cdot,\x^j) \leq \phi_{k_j-\nu_j+1}(\cdot,\x^j)$, hence we obtain
$$
f(\z^{k_j-\nu_j}) \leq \phi_{k_j-\nu_j+1}(\z^{k_j-\nu_j},\x^j) + \epsilon_j \|\z^{k_j-\nu_j}-\x^j\|.
$$
Let 
$\widetilde{\g}_j \in \partial\left(\phi_{k_j-\nu_j}(\cdot,\x^j)+i_C  \right)(\y^{k_j-\nu_j})$ denote the aggregate subgradient at $y^{k_j-\nu_j}$.
By Lemma \ref{est3}  we have 
$f(\x^j)-\Phi_{k_j-\nu_j}(\z^{k_j-\nu_j},\x^j) \geq \sigma \|\widetilde{\g}_j\| \| \x^j-\z^{k_j-\nu_j}\|$ for a
constant $\sigma$ independent of $j$. 
Now recall that $\x^j\to \x^*$  and that we
have to show that $\x^*$ is critical. It suffices to show that there is a subsequence $j\in J'$ with $\widetilde{\g}_j\to 0$.

Assume on the contrary  that $\|\widetilde{\g}_j\|\geq \eta > 0$ for every $j\in J$. Then
\[
f(\x^j) - \Phi_{k_j-\nu_j}(\z^{k_j-\nu_j},\x^j) \geq \eta\sigma \|\z^{k_j-\nu_j}-\x^j\|.
\]
Now 
expanding the test quotient at $\z^{k_j-\nu_j}$ gives
\begin{align*}
\widetilde{\rho}_{k_j-\nu_j} &= \rho_{k_j-\nu_j}+ \frac{f(\z^{k_j-\nu_j}) - \phi_{k_j-\nu_j+1}(\z^{k_j-\nu_j},\x^j)}{f(\x^j)-\Phi_{k_j-\nu_j}(\z^{k_j-\nu_j},\x^j)}\\
&\leq
\rho_{k_j-\nu_j}+ \frac{\epsilon_j \|\z^{k_j-\nu_j}-\x^j\|}{\eta\sigma \|\z^{k_j-\nu_j}-\x^j\|} < \widetilde{\gamma}
\end{align*}
for $j\in J$ sufficiently large, contradicting $\widetilde{\rho}_{k_j-\nu_j} \geq \widetilde{\gamma}$. This shows that
there must exist a subsequence $J'$ of $J$ such that $\widetilde{\g}_j\to 0$, $j\in J'$. Passing to the limit
$j\in J'$, we use the argument of part 2) to show that
$0\in \partial \left( \phi(\cdot,\x^*)+i_C \right)(\x^*)$, hence $\x^*$ is critical for (\ref{program}).
\\

7)
Now consider an infinite subsequence $j\in J$ where $\x^j\to \x^*$,  the trust-region radius $R_{k_j}$ was active
at $\tilde{\x}^{j+1}$ with non-zero multiplier $\lambda_j > 0$ when $\x^{j+1}$ was accepted, where $R_{k_j}\to 0$,
but where during the $j^{\rm th}$ inner loop the trust-region radius was never reduced. 
Since $R_{k_j}\to 0$, the work to bring the radius to
0 must be  put about somewhere else outside $J$. For every $j\in J$ define $j' \in \mathbb N$ to be the largest index $j' < j$
such that in the $j'$th inner loop, the trust-region radius was reduced at least once.
This means that in none of the loops $j'+1,\dots,j$ was the trust-region radius reduced.  As a consequence,
$$R_{k_{j'}} \leq R_{k_{j'+1}} \leq \dots \leq R_{k_j} \to 0,$$ so all the trust region radii at acceptance in any of the loops between
$j'$ and $j$ tend to 0. Indeed, all that can happen is that due to good acceptance the trust-region radii are increased at acceptance (see step 8), so
the largest among them is $R_{k_j}$. 

Let $J' = \{j': j\in J\}$, where we understand $j\mapsto j'$ as a function. Passing to a subsequence of $J,J'$,
we may assume that $\x^{j'}\to \x'$ and $\g_{j'}^* \to 0$, the latter because the sequence $J'$ corresponds to one of
the cases discussed in parts 2) - 5). Passing to yet another subsequence, we may arrange that the sequences
$J,J'$ are interlaced. That is, $j' < j < j'^{+} < j^+ < j'^{++} < j^{++} < \dots \to \infty$. This  since
$j'$ tends to $\infty$ as a function of $j$.

Now assume that there exists $\eta > 0$ such that $\|\g_j^*\|\geq \eta$ for all $j\in J$. Then since $\x^j \to \x^*$,
we also have $\x^{j+1}\to \x^*$ due to $R_{k_j}\to 0$ and the second part of  rule (\ref{trial}).
Fix $\epsilon > 0$ with $\epsilon < \eta$. For $j\in J$ large enough we have $\|\g_{j'}^*\| < \epsilon$, because
$\g_{j'}^*\to 0$, $j'\in J'$, and as $j$ gets larger, so does $j'$. That means in the interval $[j',j)$ there exists an index
$j'' \in \mathbb N$ such that
\[
\|\g_{j''}^*\| < \epsilon, \quad \| \g_{i}^*\| \geq \epsilon \,\mbox{ for all } i = j'' + 1,\dots,j.
\]
The index $j''$ may coincide with $j'$, it might also be larger, but it precedes $j$. In any case, $j \mapsto j''$ is again a function on $J$ and defines another infinite index set
$J''$ still interlaced with $J$.

Now since $|\x^i-\x^{i+1}|_{Q_i}\to 0$, and at the same time $\|\g_i^*\|\geq \epsilon$  for $i=j''+1,\dots,j$,   we can use Lemma \ref{est3} to get estimates of the form
$\sigma \|\g_i^*\| \|\x^i - \x^{i+1}\| \leq f(\x^{i})- f(\x^{i+1})$ with $\sigma$ independent of $i\in [j''+1,\dots,j]$ and $j$ large enough. Summation gives
\[
\sigma \sum_{i=j''+1}^{j} \|\g_i^*\| \|\x^i - \x^{i+1}\| \leq  f(\x^{j''+1})-f(\x^{j+1})  \to 0 \qquad (j\in J, j\to\infty, j\mapsto j'').
\]
Since $\|\g_i^*\| \geq \epsilon$ for all $i \in [j''+1,\dots,j]$,  the sequence
$\sum_{i=j''+1}^j \|\x^i-\x^{i+1}\|\to 0$ converges as $j\in J, j\to \infty$.   By the triangle inequality,
$\x^{j''+1}-\x^{j+1}\to 0$. Therefore $\x^{j''+1}\to \x^*$, and also $\tilde{\x}^{j''+1}\to \x^*$ because $\| \x^{j''+1}-\tilde{\x}^{j''+1}\| \leq 
\|\x^{j''+1} - \x^{j''}\| + \|\x^{j''} - \tilde{\x}^{j''+1}\| \leq \Theta R_{k_{j''}} + R_{k_{j''}} \to 0$.

Since  as an aggregate subgradient 
$\g^*_{j''}\in \partial \left(  \phi(\cdot,\x^{j''}) + i_C\right)(\tilde{\x}^{j''+1})$, 
passing to yet another subsequence and using local boundedness of $\partial f$,
we get $\g_{j''}^* \to {\g}^*$. But with $x^{j''}\to \x^*$ and model property $(M_3)$
we get $g^*\in  \partial(\phi(\cdot,\x^*)+i_C)(\x^*)$.  
It follows that $\partial (\phi(\cdot,\x^*)+i_C)(\x^*)$ contains an element $\g^*$ of norm less than or equal to $\epsilon$.  
As $\epsilon < \eta$ was arbitrary, we conclude that $0\in\partial (\phi(\cdot,\x^*)+i_C)(\x^*) \subset  \partial f(\x^*)+N_C(\x^*)$.
That settles the remaining case.
\end{proof}

\begin{remark}
As the proof reveals, strictness of the cutting plane oracle is not needed if the trust region radius
stays bounded away from 0, or if the trust region constraint is not strongly active.
\end{remark}



\section{Applications}
\label{applications}
A natural question is why the secondary test
(\ref{secondary}) in step 7 of the algorithm is required. A partial answer is that  convergence fails if the test is removed without substitute, as shown in our examples in \cite{loja,ANR:2016}.
In exchange, sometimes
the secondary test (\ref{secondary})  becomes redundant. There are at least three
situations where this happens.

The first case is of course the convex tangent plane oracle $\mathscr O^{\rm tan}$, where cutting planes are tangents to $f$ and $Q=0$. Here  $\rho_k = \widetilde{\rho}_k$,
so for a null step the trust region radius is fixed. This is  corroborated by \cite{rus}, where convergence is proved with
a  trust region radius fixed all along. The resulting method then resembles the cutting plane method.

\subsection{The all-tangents-oracle}

A second case is
when a model $\phi$ is available to construct an oracle  and we choose the working model as $\phi_k = \phi$ at all $k$.  This is allowed and
corresponds to choosing as oracle $\mathscr O^{\rm all} (\z,\x)$ the set of {\em all tangents to}  $\phi(\cdot,\x)$ at all points $\z' \in B(\x,M)$. Then always $\widetilde{\rho}_k \geq 1$,
hence the trust region radius is {\em always} reduced at a null step.  The price to pay for this simplification is that the tangent program is no longer of a simple structure.

\begin{example}
{\bf (Proximal point method)}.
Let us show that splitting techniques  fit nicely into the framework of the all-tangents-oracle $\mathscr O^{\rm all}$.
Consider $f = g + h$ with $g$ convex and $h$ of class $C^1$. We know that
$\phi(\y,\x) = g(\y) + h(\x)+\nabla h(\x) (\y-\x)$ is a strict model of $f$, because here the Clarke subdifferential is additive
due to strict differentiability of $h$. Now we use $\mathscr O^{\rm all}$, that is, we choose
as  working model $\phi_k(\cdot,\x)=\phi(\cdot,\x)$.  
As second order working model 
we choose 
\[
\Phi(\y,\x) = \phi(\y,\x) + \textstyle\frac{1}{2r_\x} \|\y-\x\|^2,
\]
where $r_\x > 0$ may depend on  the serious iterate $\x$. That corresponds to $Q(\x) = (1/r_\x) I$. 
Note that if we add the constant term $(r_\x/2 )\|\nabla h(\x)\|^2-h(\x)$, then the tangent program
(\ref{tangent}) becomes
\begin{equation}
\label{tangent_prox}
\x^+= \argmin_\y\Phi(\y,\x) -h(\x) + \textstyle\frac{r_\x}{2}\|\nabla h(\x)\|^2 = \displaystyle\argmin_\y g(\y) + \textstyle\frac{1}{2r_\x} \| \y-\x+r_\x \nabla h(\x) \|^2
\end{equation}    
which  is the convex proximal point method with splitting. This means convergence theory of the convex proximal point method
with splitting is a very special case of our convergence theory.  Convexity of $h$ is not required in our approach. If $h$ is in addition convex, convergence to
a single minimum follows with Kiwiel's anchor technique \cite{kiwiel_ancre}, see also \cite{Kiw:85,kiwiel:86}.
\end{example}

\begin{example}
{\bf (Forward-backward splitting)}.
In contrast with the typical splitting literature we  do not require convexity of $g$. If  $g$ is  lower $C^2$, so that
$g + (1/2\ell_\x) \|\cdot - \x\|^2$ is convex on $B(\x,R)$ for a suitable $\ell_\x > 0$, then
we choose as first-order model
$\phi(\y,\x) = g(\y) + (1/2\ell_\x)\|\y-\x\|^2 +h(\x)+ \nabla h(\x) (\y-\x)$. Now for $r_\x < \ell_\x$ we use the second order working model
\[
\Phi(\y,\x) = \phi(\y,\x) + (1/2r_\x-1/2\ell_\x)\|\y-\x\|^2,
\]
then on adding the same constant term $-h(\x)+(r_\x/2) \|\nabla h(\x)\|^2$, and using $\mathscr O^{\rm all}$, we end up with the same tangent program (\ref{tangent_prox}), 
where now $r_\x< \ell_\x$ is required to convexify $g$ in the neighborhood $B(\x,R)$ of the current $\x$ in which trial steps are
taken, assuring a convex tangent program. 

But we can do still better, because if $g$ is only lower $C^1$, we can still use the oracle $\mathscr O^\downarrow$ for $g$ and the standard model for $h$,
so that a cutting plane is of the form $m^\downarrow(\cdot,\x) + h(\x)+ \nabla h(\x)(\cdot-\x)$. This oracle is strict. 
\end{example}

\begin{remark}
In \cite{fort,davis} the proximal point method is combined with uncertainty in the computation
of $f,\nabla h$, respectively, of $\partial g$, where Monte-Carlo, respectively, stochastic subgradient techniques cause the uncertainty. 
This now turns out a special case of
our approach \cite{inexact} to inexact subgradients and values  in the non-convex bundle method. In that work
we also allow the more challenging situation when the level of uncertainty in the data
remains unknown to the user. 
\end{remark}

\subsection{The standard model}
Let us consider the standard oracle $\mathscr O^\sharp = \mathscr O_{\phi^\sharp}$, where  
the cutting plane $(a,\g)\in \mathscr O^\sharp(\z,\x)$ requires finding the
$\g\in \partial f(\x)$ with $f^\circ(\x,\z-\x) = \g^\ttop  (\z-\x)$. While this may still be hard to compute in some cases,
we get a simplification if $\x$ is a point of strict differentiability of $f$, as then $\partial f(\x)=\{\nabla f(\x)\}$. In that case
$\phi_k(\z^k,\x)=\phi^\sharp(\z^k,\x)$, and the secondary test is again redundant. See Borwein and Moors \cite{borwein_densely_strict}
for criteria when a function $f$ is densely or almost everywhere strictly differentiable. 

In that case, we automatically have $\phi_k = \phi^\sharp$, so the secondary test is again redundant. 
We have then arranged that the trust-region tangent program is
\[
\min\{f(\x^j) + \nabla f(\x^j)(\y-\x) + \textstyle\frac{1}{2}(\y-\x^j)^\ttop Q_j (\y-\x^j): \|\y-\x^j\|\leq R, \y\in C\},
\]
which is the classical trust-region method. Convergence of this method, which the user cannot distinguish from classical
trust regions,  hinges therefore on whether
$\mathscr O^\sharp$ is strict at the accumulation points of the sequence $\x^j$. As we know, this is almost never the case
in non-smooth optimization, which gives us the explanation why classical methods as a rule do not work on non-smooth problems.

\end{document}